\documentclass[a4paper,10pt]{article}
\usepackage[utf8]{inputenc}
\usepackage[a4paper, total={7in, 9in}]{geometry}
\usepackage{microtype}
\usepackage{amssymb}
\usepackage{amsmath}
\usepackage{subcaption}
\usepackage{graphicx,amsthm,color}
\usepackage{diagbox}
\usepackage{float}
\usepackage{hyperref,cleveref,cite,url}
\hypersetup{
  colorlinks,
  citecolor=green!50!black,
  linkcolor=red,
  urlcolor=blue!50!black}
\usepackage{makecell}
\usepackage{pgf,tikz}
\usetikzlibrary{arrows}
\usepackage{authblk}
\usepackage{xspace}
\usepackage{enumerate}
 
\newtheorem{theorem}{Theorem}

\newtheorem{proposition}[theorem]{Proposition}

\newtheorem{corollary}[theorem]{Corollary}
\newtheorem{lemma}[theorem]{Lemma}
\newtheorem{claim}[theorem]{Claim}
\newtheorem{case}{Case}
\theoremstyle{definition}
\newtheorem{remark}[theorem]{Remark}
\newtheorem{example}[theorem]{Example}
\newtheorem{definition}[theorem]{Definition}
\newtheorem{conjecture}[theorem]{Conjecture}

\DeclareMathOperator{\VCdim}{dim_{VC}}
\DeclareMathOperator{\edim}{dim_{\mathbb{E}}}
\DeclareMathOperator{\cdim}{cdim}
\DeclareMathOperator{\ext}{ex}
\DeclareMathOperator{\conv}{conv}
\newcommand{\M}{\ensuremath{\mathcal{M}}}

\newcommand{\covectors}{\ensuremath{\mathcal{L}}}

\newcommand{\topes}{\ensuremath{\mathcal{T}}}
\DeclareMathOperator{\Sep}{\mathrm Sep}

\renewcommand{\SS}{\ensuremath{\mathcal{S}}}
\newcommand{\Q}{\mathrm Q}

\newcommand{\cA}{\mathcal{A}\xspace}

\newcommand{\cC}{\mathcal{C}\xspace}

\newcommand{\cF}{\mathcal{F}\xspace}
\newcommand{\cH}{\mathcal{H}\xspace}
\newcommand{\cI}{\mathcal{I}\xspace}

\newcommand{\cM}{\mathcal{M}\xspace}
\newcommand{\cR}{\mathcal{R}\xspace}
\newcommand{\cP}{\mathcal{P}\xspace}
\newcommand{\cO}{\mathcal{O}\xspace}

\DeclareMathOperator{\downarr}{\downarrow}
\DeclareMathOperator{\uparr}{\uparrow}

\newcommand{\R}{\mathbb{R}\xspace}
\newcommand{\uX}{\ensuremath{\underline{\mathcal{X}}}\xspace}
\newcommand{\oX}{\ensuremath{\overline{\mathcal{X}}}\xspace}
\DeclareMathOperator{\vcd}{VC-dim}

\newcommand{\oldqed}{}
\def\endofFact{\hfill$\Diamond$}
\newenvironment{claimproof}{
  \renewcommand{\oldqed}{\qed}
  \renewcommand{\qed}{\endofFact}
  \begin{proof}
  \leftskip15pt\relax
}{
  \end{proof}
  \renewcommand{\qed}{\oldqed}
} 

\title{Geometry of convex geometries}
\author{J\'er\'emie Chalopin, Victor Chepoi, and Kolja Knauer}
\date{\today}

\begin{document}

\maketitle

\begin{abstract} We prove that any convex geometry $\cA=(U,\cC)$ 
on $n$ points and any ideal $\cI=(U',\cC')$ of $\cA$ can be realized as the intersection pattern 
of an open convex polyhedral cone $K\subseteq {\mathbb R}^n$ with the orthants of ${\mathbb R}^n$. 
Furthermore, we show that $K$ can be chosen to have at most $m$ facets, 
where $m$ is the number of critical rooted circuits of $\cA$. We also show that any convex geometry of 
convex dimension $d$ is realizable in ${\mathbb R}^d$ and that any multisimplicial complex 
(a basic example of an ideal of a convex geometry) of dimension $d$ is realizable in 
${\mathbb R}^{2d}$ and that this is best possible. From our results it also follows 
that distributive lattices of dimension $d$ are realizable 
in ${\mathbb R}^{d}$ and that median systems are realizable. We leave open 
whether each median system of dimension 
$d$ is realizable in ${\mathbb R}^{O(d)}$. 
\end{abstract}

\section{Introduction}
Several fundamental combinatorial structures constitute abstract generalizations of geometric settings. Matroids generalize the linear independence in vector spaces, oriented matroids (OMs) capture the combinatorics of  regions in a central hyperplane arrangement in ${\mathbb R}^d$, ample/lopsided sets (AMPs) encode the regions of the arrangement of coordinate hyperplanes intersected with a convex set, and  convex geometries/antimatroids represent an abstraction of Euclidean convexity restricted to a finite set. Finally, complexes of oriented matroids (COMs) are a common generalization of oriented matroids and ample sets and capture  the combinatorics of  regions in an arbitrary hyperplane arrangement in ${\mathbb R}^d$ restricted to a convex set $K$. Although this geometric model is a desirable property for a respective combinatorial structure, the realizability question is hard. For example, the problem of characterizing which oriented matroids come from hyperplane arrangements is intractable~\cite{Mne88,Ric95,Sho91} and the realizability of a convex geometry in ${\mathbb R}^2$ is $\exists{\mathbb R}$-complete  \cite{AW10,HoMe}. All the above structures can be seen as set systems, see~\Cref{setsystem}. This  allows to view convex geometries as ample sets and ample sets as COMs.

We investigate the realizability question for convex
geometries and generalize them to ideals of convex geometries. Convex geometries (alias antimatroids)
have been introduced and investigated by Edelman and Jamison
\cite{EdJa} in the context of abstract convexity and by Korte and
Lovasz \cite{KoLo,KoLoSch} in combinatorics. Kashiwabara, Nakamura, and Okamoto \cite{KaNaOk}  proved that any convex geometry 
 $\cA=(U, \cC)$ on $n$ points can be realized in ${\mathbb
  R}^n$ using a  generalized convex shelling. Richter and Rogers \cite{RiRo} proved that $\cA$  can be realized in this way in ${\mathbb R}^d$, where $d$ is the convex dimension of $\cA$. In this paper, we consider a simpler (and dual) version of realizability  via hyperplane arrangements and convex sets, as in the case of OMs and COMs. We prove that any convex
geometry $\cA$  is realizable in this way in ${\mathbb
  R}^n$. Furthermore, we show that the convex set 
realizing $\cA$ can be chosen to be a polyhedral cone with at most $m$
facets, where $m$ is the number of critical rooted  circuits of $\cA$.  We
also establish that any ideal $\cI=(U,\cC')$ of $\cA$ is realizable in
${\mathbb R}^n$ by a convex polyhedron with $m+k$ facets where $k$ is
the number of positive circuits. We also show that any convex geometry of convex dimension $d$ is realizable in ${\mathbb R}^d$. 
As an application of our results on ideals, we show that any multisimplicial complex  of dimension $d$ is realizable in ${\mathbb R}^{2d}$ and this is optimal. It follows that distributive lattices of dimension $d$ are realizable in ${\mathbb R}^{d}$ and that median systems are realizable. We leave open whether each median system of dimension $d$ is realizable in ${\mathbb R}^{O(d)}$ but show that any tree (median system of dimension 1) is realizable in ${\mathbb R}^{2}$. 

\section{Preliminaries} 
In this section, we define the main combinatorial structures investigated in this paper and their realizability. 

\subsection{Set families and systems of sign vectors} 

Let $U$ be a set of size $n$. A \emph{set family} $\SS$ is any collection  of subsets of $U$. We denote by $\SS^*$ the complement 
 $2^U\setminus \SS$ of the family $\SS$. 
Any set family $\SS\subseteq 2^U$ can be viewed as a subset of
vertices of the $n$-dimensional hypercube $\Q_n=\Q(U)$. Denote by $G(\SS)$ the subgraph of $\Q_n$ induced
by the vertices of $\Q_n$ corresponding to the sets of $\SS$; $G(\SS)$ is called the \emph{1-inclusion
graph} of $\SS$. A set family $\SS$ is called \emph{isometric} if $G(\SS)$ is an \emph{isometric subgraph} of the hypercube $\Q_n$, i.e., 
the distances in $G(\SS)$ and in $\Q_n$ between any two vertices of $G(\SS)$ are equal. 
An {\it $X$-cube} $Q$ of $\Q_n$ is the 1-inclusion graph
of the set family  $\{ Y\cup X': X'\subseteq X\}$, where $Y$ is a subset 
of $U\setminus X$, called the \emph{support} of $Q$. If $|X|=n'$, then any 
$X$-cube is a $n'$-dimensional subcube of $\Q_n$ and $\Q_n$ contains $2^{n-n'}$ 
$X$-cubes. 

Let $\covectors$ be a {\it system of
sign vectors} on $U$, i.e., maps from $U$ to $\{-1,0,+1\}$. The elements of
$\covectors$ are referred to as \emph{covectors} and denoted by capital letters
$X, Y, Z$. For $X \in \covectors$, the subset $\underline{X} = \{e\in U:
X_e\neq 0\}$ is the \emph{support} of $X$ and  its complement
$X^0=U\setminus \underline{X}=\{e\in U: X_e=0\}$ is the \emph{zero set} of $X$. 

The systems of sign vectors generalize set families.  
Indeed, any set family $(U,\SS)$ can be encoded as a system of sign vectors by setting for each set $X\in \SS$, $X_e=-1$ if $e\notin X$ and $X_e=+1$ if $e\in X$. In this representation, each set $X$ is encoded by a $\{-1,+1\}$-vector. The encoding with  $\{-1,0,+1\}$-vectors is useful when encoding cubes of $Q_n$. Indeed, each $X$-cube $Q$ with support $Y$ can be encoded by $\{-1,0,+1\}$-vector $X(Q)$, where $X(Q)_e=+1$ if $e\in Y$, $X(Q)_e=-1$ if $e\in (U\setminus X)\setminus Y$, and $X(Q)_e=0$ if $e\in X$. 

For each subset $X\subset U$, the (open) \emph{$X$-orthant} of ${\mathbb R}^n$ is the set $\cO(Y)$ of all points $x=(x_1,\ldots,x_n)\in {\mathbb R}^n$ 
such that $x_e>0$ if $e\in X$ and $x_e<0$ if $e\notin X$. More generally, for a covector $X\in \{ -1,0,+1\}^U$, the \emph{$X$-generalized orthant} is the set of all points $x=(x_1,\ldots,x_n)\in {\mathbb R}^n$  such that $x_e>0$ if $X_e=+1$ and $x_e<0$ if $X_e=-1$. Notice that 
the $X$-generalized orthant is the union of all $Y$-orthants $\cO(Y)$ such that $Y=\underline{X}\cup Y'$ with $Y'\subseteq X^0$.

\subsection{OMs and COMs} 
We recall the basic theory of OMs
and COMs from~\cite{BjLVStWhZi} and~\cite{BaChKn}, respectively. 
Co-invented by Bland and Las Vergnas~\cite{BlLV}
and Folkman and Lawrence~\cite{FoLa}, and further investigated  by Edmonds and Mandel~\cite{ed-ma-82} and
many other authors, oriented matroids (OMs) represent a unified combinatorial theory of orientations of ordinary matroids,
which simultaneously captures the basic properties of sign vectors representing the regions in a hyperplane
arrangement  in ${\mathbb R}^d$ and of  sign vectors of the circuits  in a directed
graph. OMs provide a framework for the analysis of
combinatorial properties of geometric configurations occurring  in discrete
geometry and in machine learning. Point and vector configurations, order types,
hyperplane and pseudo-line arrangements, convex polytopes, directed graphs, and
linear programming find a common generalization in this language. The
Topological Representation Theorem of~\cite{FoLa} connects the theory of OMs on a
deep level to arrangements of pseudohyperplanes and distinguishes it from the
theory of ordinary matroids. 

%
Let $\covectors\subseteq \{ -1,0,+1\}^U$ be a system of sign vectors. For the sake of this paper we assume that $\covectors$ is \emph{simple}, {i.e., $\{X_e:  X\in \covectors\}=\{-1,0,+1\}$.} 
%
For $X,Y\in \covectors$, $\Sep(X,Y)=\{e\in U: X_eY_e=-1\}$ is the
\emph{separator} of $X$ and $Y$. The \emph{composition} of $X$ and $Y$ is the
sign
vector $X\circ Y$, where for all $e\in U$,
$(X\circ Y)_e = X_e$ if $X_e\ne 0$  and  $(X\circ Y)_e=Y_e$ if $X_e=0$.

\begin{definition} \label{def:OM}
	An \emph{oriented matroid} (OM)~\cite{BjLVStWhZi} is a system $\cM = (U,\covectors)$ of sign vectors
	 satisfying
	\begin{itemize}
		\item [{\bf (C)}] ({\sf Composition)} $X\circ Y \in \covectors$ for all
		$X,Y \in \covectors$.
		\item[{\bf (SE)}] ({\sf Strong elimination}) for each pair
		$X,Y\in\covectors$ and for each $e\in \Sep(X,Y)$, there exists $Z \in
		\covectors$ such that $Z_e=0$ and $Z_f=(X\circ Y)_f$ for all $f\in
		U\setminus \Sep(X,Y)$.
		\item[{\bf (Sym)}] ({\sf Symmetry}) $-\covectors=\{ -X: X\in
		\covectors\}=\covectors,$ that is, $\covectors$ is closed under sign
		reversal.
	\end{itemize}
\end{definition}

Complexes of Oriented Matroids (COMs) were introduced
by Bandelt, Chepoi, and Knauer~\cite{BaChKn} as a natural common generalization of ample sets and OMs. 
They are defined by replacing the global axiom (Sym) with a weaker local axiom:

\begin{definition} \label{def:COM}
	A \emph{complex of oriented matroids} (COM)~\cite{BaChKn} is a system of sign vectors $\cM=(U,\covectors)$
	 satisfying (SE)  and the following axiom:
	\begin{itemize}
		\item[{\bf (FS)}] ({\sf Face symmetry}) $X\circ -Y \in  \covectors$
		for all $X,Y \in  \covectors$.
	\end{itemize}
      \end{definition}

One can  see that OMs are exactly the COMs containing the zero vector ${\bf 0}$,
see~\cite{BaChKn}. Let $\leq$ be the product ordering on $\{-1,0,+1\}^{U} $
relative to  $0 \leq -1, +1$.
The poset $(\covectors,\le)$ of a COM $\cM =(U,\covectors)$ with an artificial maximum
$\hat{1}$ forms the (graded)
\emph{big face (semi)lattice} $\mathcal{F}_{\mathrm{big}}(\cM)$. For $X\in\covectors$ a covector of a COM 
$\cM=(U,\covectors)$, the
\emph{face} of $X$ is $\uparr X:=\{Y\in\covectors: X\leq Y\}$, see~\cite{BaChKn,BjLVStWhZi}. 
By \cite[Lemma 4]{BaChKn}, each face $\uparr X$ of a COM $\cM$ is an OM.

The \emph{topes} $\topes$ of $\covectors$ are the co-atoms of
$\mathcal{F}_{\mathrm{big}}(\cM)$. By simplicity the topes are  $\{-1,+1\}$-vectors 
and  ${\mathcal T}$ can be seen as a family of subsets of $U$. 
For each $T\in\mathcal{T}$, an element
$e\in U$ belongs to the corresponding
set if and only if $T_e=+1$. Since by~\cite{BaChKn,KnMa} the topes determine $\cM$ uniquely, we make frequent use of the following:

\begin{remark}\label{setsystem}
    Every simple COM $\cM$ can be seen as the set family, defined by its set of topes $\mathcal{T}$.
\end{remark}

The \emph{tope graph} $G(\cM)$ of a COM $\cM = (U,\covectors)$ is the
1-inclusion graph of the set $\mathcal T$ of topes of $\covectors$,
i.e., the subgraph of the hypercube induced by the vertices
corresponding to $\topes$. The COM $\cM = (U,\covectors)$ can
be recovered up to relabelling and reorientation from $G(\cM)$ and $G(\cM)$ is an isometric 
subgraph of $\Q_n$, see~\cite{BaChKn,KnMa}.



%

\subsection{Ample sets}  Ample sets~\cite{BaChDrKo} (originally introduced as lopsided sets by Lawrence~\cite{La})  are combinatorial structures somewhat opposed to oriented matroids.  They capture an important variety of combinatorial objects, e.g.,  diagrams of (upper locally) distributive lattices, median graphs or CAT(0) cube complexes, convex geometries and conditional antimatroids, see \cite{BaChDrKo}. 
\emph{Ample sets}   
are exactly the COMs, in which all faces are cubes. 
Ample sets can be defined and characterized in a multitude of combinatorial ways in the language of set families, i.e., in terms of topes; see \cite{BaChDrKo,BoRa,La}. One of them is via shattering and strong shattering.  

Let $\SS$ be a family of subsets of an $n$-element  set $U$.  
For a set $Y\subset U$, the \emph{trace} of $\SS$ to $Y$ is defined as $\SS|_Y=\{ X\cap Y: X\in \SS\}$. A subset $X$ of $U$ is \emph{shattered} by $\SS$ if for
all $Y\subseteq X$ there exists $S\in\SS$ such that $S\cap X=Y$, i.e. $\SS_X=2^X$. The \emph{Vapnik-Chervonenkis dimension}
(the \emph{VC-dimension} for short)  $\vcd(\SS)$ of $\SS$ is the cardinality of the largest subset of $U$ shattered by $\SS$.
A subset $X$ of $U$ is \emph{strongly shattered} by $\SS$ if the 1-inclusion graph $G(\SS)$ of $\SS$ contains
an $X$-cube.  Denote by $\oX(\SS)$ and $\uX(\SS)$ the families  consisting of all shattered and
of all strongly shattered sets of $\SS$, respectively.  Clearly,
$\uX(\SS)\subseteq \oX(\SS)$ and both $\oX(\SS)$ and $\uX(\SS)$ are closed under
taking subsets, i.e., $\oX(\SS)$ and $\uX(\SS)$ are \emph{abstract simplicial complexes}.
The VC-dimension  $\vcd(\SS)$ of  $\SS$ is thus the size of a
largest set shattered by $\SS$, i.e., the dimension of the simplicial
complex $\oX(\SS)$. A family  $\mathcal S$ of subsets of $U$ is \emph{ample} whenever the simplicial complexes 
 $\oX(\SS)$ and $\uX(\SS)$ coincide, i.e., each shattered set is strongly shattered. From this definition it follows that 
 the VC-dimension of an ample set $\SS$ coincides with the dimension of a largest cube included in $\SS$. The complement  $\SS^*=2^U\setminus \SS$ of an ample set $\SS$ is also ample. Finally, ample sets are isometric. 

 \subsection{Convex geometries} Convex geometries, introduced and investigated by Edelman and Jamison~\cite{EdJa},  are the abstract convexity spaces satisfying  one of the most important properties of Euclidean convexity: each convex set is the convex hull of its extremal points. Convex geometries are a particular case of ample sets~\cite{BaChDrKo}.
 A \emph{convex geometry} \cite{EdJa} is a pair $\cA=(U,\mathcal{C})$ of a finite \emph{universe} $U$ and a collection of \emph{convex sets} $\mathcal{C}\subseteq 2^U$ satisfying the following three conditions:
\begin{itemize}
    \item[(C1)] $\varnothing\in \mathcal{C}$ and $U\in \mathcal C$;
    \item[(C2)] $X\cap Y\in \mathcal{C}$ for all $X,Y\in \mathcal{C}$;\hfill (intersection-closed)
    \item[(C3)] if $X\in {\mathcal C}\setminus \{ U\}$, then there exists $e\in U\setminus X$ such that $X\cup e\in \mathcal{C}$. \hfill (extendable)
\end{itemize}

See the left side of~\Cref{fig:convsemigeo} for an example.
The \emph{convex hull} $\conv(A)$ of a set $A\subset U$ is the
intersection of all convex sets of $\cA$ containing $A$; by (C2),
$\conv(A)$ is the smallest element of $\cC$ containing $A$.  A point
$x$ of a convex set $X$ is called an \emph{extreme point} of $X$ if
$X\setminus \{ x\}$ is also convex. Denote by $\ext(X)$ the set of all
extreme points of $X$. Convex geometries can be also characterized by
(C1), (C2), and the following axiom:
\begin{itemize}
    \item[(C4)] $X=\conv(\ext(X))$ for any $X\in\cC$.
\end{itemize}
Convex geometries can be also characterized by (C1), (C2) and the \emph{anti-exchange} axiom: 
\begin{itemize}
    \item[(C5)] if $X\in \mathcal{C}$ and $p,q$ are two different points of $U\setminus X$, then $q\in \conv(X\cup p)$ implies that $p\notin \conv(X\cup q)$. 
\end{itemize}
For other characterizations of convex geometries, their properties, and examples, see the foundational paper by Edelman and Jamison \cite{EdJa}. 
Convex geometries are ample sets \cite{BaChDrKo}.  
If $\cA=(U,\mathcal{C})$ is a convex geometry, then the set family $\cA^*$ on $U$ consisting of the complements $U\setminus C$ of sets $C$ of $\mathcal{C}$ is union closed and is called an \emph{antimatroid}.  For a nice exposition of the theory  of convex geometries/antimatroids, see the book by Korte, Lov\'asz, and Schrader \cite{KoLoSch}.  Below we will present a characterization of convex geometries via rooted circuits and the notion of convex dimension. 

\subsection{Bouquets and ideals of convex geometries}

We now consider a generalization of ideals of convex geometries,
originally introduced under the name of conditional antimatroids
in~\cite{BaChDrKo}. Analogously to convex geometries bouquets of
convex geometries are also ample sets~\cite{BaChDrKo}.  A pair
$\cA=(U,\mathcal{C})$ of a finite \emph{universe} $U$ and a collection
of \emph{convex sets} $\mathcal{C}\subseteq 2^U$ is called
\emph{bouquet of convex geometries} if it satisfies the following two
conditions:
\begin{itemize}
    \item[(C1$'$)] $\varnothing\in \mathcal{C}$;
    \item[(C2)] $X\cap Y\in \mathcal{C}$ for all $X,Y\in \mathcal{C}$;
    \item[(C3$'$)] for all $X,Y\in \mathcal{C}$ with $Y\subset X$, there is $e\in X\setminus Y$ such that $Y\cup e\in \mathcal{C}$. \hfill (locally extendable)
\end{itemize}

A pair $\cI=(U',\cC')$ is called an  \emph{ideal of a convex geometry} if there is a convex geometry $\cA=(U,\cC)$ such that $U'\subseteq U$, $\cC'\subseteq \cC$ and if $X\in \cC, Y\in \cC'$ and $X\subseteq Y$, then $X\in \cC'$. Equivalently, for any $X\in \cC'$, its principal ideals 
in $\cC'$ and $\cC$ coincide, where the \emph{principal ideal} $\downarr X$ of $X$ in $\cC$ (respectively, in $\cC'$) is the set of all subsets $Y \subseteq X$ such that $Y \in \cC$ (respectively, $Y \in \cC'$).  
It is easy to see that intervals of bouquets of convex geometries are
convex geometries and that ideals of convex geometries are bouquets
of convex geometries. However, there are bouquets of convex geometries
that are not ideals.

\begin{figure}[htp]
    \centering
    \includegraphics[width=.8\textwidth]{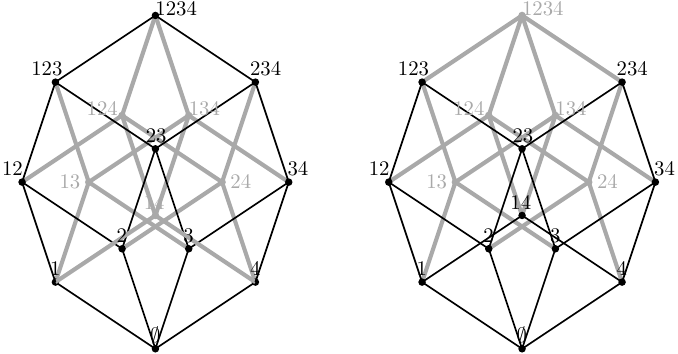}
    \caption{Left: A convex geometry. Right: A bouquet of convex geometries
      that is not an ideal of a convex geometry. Both are represented in black
       as a subgraph of the 4-dimensional hypercube (in gray).
    }\label{fig:convsemigeo}
\end{figure}

\begin{proposition} \label{nonideal}
There exist bouquets of convex geometries that are not ideals of convex geometries.  
\end{proposition}
\begin{proof}
  Consider the bouquet of convex geometries in~\Cref{fig:convsemigeo},
  right, i.e., $U'=\{1,2,3,4\}$ and
  \[\mathcal{C}'=\{\varnothing, \{1\}, \{2\}, \{3\}, \{4\}, \{1,2\},
    \{1,4\}, \{2,3\}, \{3,4\}, \{1,2,3\}, \{2,3,4\}\}.\] Suppose that
  $\cA'=(U',\mathcal{C}')$ is an ideal of a convex geometry
  $\cA=(U,\cC)$. Starting with $\{1,4\}$ and applying iteratively
  (C3), there exists a set $Y \subseteq U \setminus U'$ such that one
  of the sets $Y \cup \{1,2,4\}$ or $Y \cup \{1,3,4\}$ must be in
  $\cC$. Now applying (C2) with respect to $Y \cup \{1,2,4\}$ and
  $\{2,3,4\}$ or $Y \cup \{1,3,4\}$ and $\{1,2,3\}$, we get that
  $\{2,4\}$ or $\{1,3\}$ are in $\cC$. Hence $\cA'=(E',\mathcal{C}')$
  is not an ideal of $\cA$.
\end{proof}




\section{Realizability}  In this section, we recall the notion of realizability of COMs and investigate some general properties of realizable COMs. We also compare the notion of realizability with the notion of generalized convex shelling, which can be seen as a dual realizability. 


\subsection{Realizability of COMs} 
We continue with the definition of realizable COMs given in the paper
\cite{BaChKn}, which generalizes realizability of oriented and affine
oriented matroids~\cite{BjLVStWhZi} and of ample sets~\cite{La}.  An
\emph{affine arrangement of hyperplanes} $\mathcal H$ is a finite set
of affine hyperplanes of ${\mathbb R}^d$.  We will denote by
${\mathcal H}_d$ the set of $d$ coordinate hyperplanes of
${\mathbb R}^d$. Let $K$ be a relatively open convex set of ${\mathbb
  R}^d$. Each (oriented) hyperplane $H=\{ x\in {\mathbb R}^d: ax=b\}$ in an
arrangement of hyperplanes $\mathcal H$ splits $\mathbb{R}^d$ into the positive part $H^+=\{ x\in {\mathbb R}^d: ax>b\}$, the
negative part $H^-=\{ x\in {\mathbb R}^d: ax<b\}$, and the zero part $H$. The arrangement $\mathcal H$
partitions ${\mathbb R}^d$ into relatively open convex regions, called
\emph{cells}. All points belonging to the same cell have the same sign
vector with respect to the hyperplanes of ${\mathcal H}$. Restrict the
arrangement pattern to $K$, that is, remove all sign vectors which
represent the regions of the partition disjoint from $K$. The resulting set of sign vectors of the cells of $K$ is denoted
$\covectors({\mathcal H},K)$ and constitutes a COM
$\cM({\mathcal H},K) = (H,\covectors({\mathcal H},K))$, see~\cite{BaChKn}.

If ${\mathcal H}$ is a central arrangement  with $K$ being any relatively open convex set containing the origin, then $\cM({\mathcal H},K)$ coincides with the notion of  \emph{realizable oriented matroid}~\cite{BjLVStWhZi}. If the  arrangement ${\mathcal H}$ is affine and $K$ is the entire space, then
$\cM({\mathcal H},K)$ coincides with the \emph{realizable affine oriented matroid}.
Finally, the \emph{realizable ample sets}  arise by taking  the central arrangement ${\mathcal H}_d$ of coordinate hyperplanes  restricted to
a full-dimensional open convex set $K$ of ${\mathbb R}^d$ (this model was first considered in \cite{La}). A COM $\cM$ is called \emph{realizable}  if there exists  an affine arrangement of hyperplanes $\mathcal H$ and a relatively open convex set $K$ of ${\mathbb R}^d$ such that $\cM=\cM({\mathcal H},K)$. Then the pair $(\cH,K)$ is called a \emph{realization} of $\cM$. 
As noticed in \cite{BaChKn}, for a realizable COM the set $K$ always can be selected to be a relatively open polyhedron. For a realizable COM $\cM$, we denote by $\edim(\cM)$ the smallest $d$ such that $\cM$ has a realization in $\mathbb{R}^d$ and call it the \emph{Euclidean dimension} of $\cM$. 

In realizable COMs, $X\le Y$ for two covectors $X,Y$  if and only if the cell corresponding to $Y$ is contained in the cell corresponding to $X$. Consequently, the topes of realizable COMs are the covectors of the inclusion maximal cells (which all have dimension $d$), called \emph{regions}. Therefore, the tope graph of a realizable COM with realization $(\cH,K)$  can be viewed as the adjacency graph of regions:
the vertices of this graph are the regions of $K$ defined by the hyperplane arrangement $\cH$ and two regions are
adjacent in this graph if they are separated by a unique hyperplane of the arrangement (recall that a hyperplane $H$ \emph{separates} two disjoint relatively open convex sets $A$ and $B$ if $A$ and $B$ belong to distinct halfspaces defined by $H$). 

To realize a COM $\cM$, it is necessary to find an arrangement of
hyperplanes $\mathcal H$, a relatively open convex set $K$, and a sign-preserving bijection
between the topes of $\cM$ and the sign vectors of the regions
(maximal cells) defined by $\mathcal H$ and included in $K$. Now, we
will show that $\mathcal H$ can be selected to be the set of coordinate
hyperplanes $\cH_d$ of ${\mathbb R}^d$ and $K$ to be a relatively open polyhedron
 of ${\mathbb R}^d$. This
immediately follows from the following lemma:

\begin{lemma}\label{cones}
Let $\mathcal{M}=(U,\covectors)$ be a realizable COM. Then $\mathcal{M}$ has a realization $(\cH'',K''),$ where $K''$ is a relatively open polyhedral cone, and $\cH''$ is a central arrangement of hyperplanes, whose normal vectors generate the ambient space.
\end{lemma}
\begin{proof}
Let $\cM = (U,\covectors)$ be a realizable COM with a realization $\cM=\cM({\mathcal H},K)$, where  ${\mathcal H}$ is an arrangement of affine hyperplanes and $K$ is a relatively open polyhedron of ${\mathbb R}^d$, as provided by~\cite{BaChKn}. 
First, we consider ${\mathbb R}^d$ as the hyperplane $x_{d+1}=1$ of ${\mathbb R}^{d+1}$. Second, we transform each hyperplane $H_e\in {\mathcal H}, e\in U$ to a central hyperplane $H'_e$ of ${\mathbb R}^{d+1}$  passing via $H_e$ and the origin of coordinates. Denote the resulting arrangement by ${\mathcal H}'$. The supporting hyperplanes $S_j$ of the closure of $K$ define a second arrangement of hyperplanes ${\mathcal S}$ of ${\mathbb R}^d$. 
Analogously to $\mathcal H$, the arrangement $\mathcal S$ can be transformed into an arrangement ${\mathcal S}'$ of central hyperplanes of  ${\mathbb R}^{d+1}$. Then  $K$ belongs to one halfspace defined by each hyperplane $S'_j\in {\mathcal S}'$. The intersection of those positive open halfspaces is an open polyhedron $K'$ containing $K$. Then  $\covectors=\covectors({\mathcal H}',K')$ holds. Each cell $C\subset K$ of the realization $\cM({\mathcal H},K)$ gives rise to an open polyhedral cone $C'\subset K'$ (containing $C$) of the realization $\cM({\mathcal H}',K')$. The cone $C'$ is the intersection of  positive or negative halfspaces defined by the hyperplanes of ${\mathcal H}'$ and by the halfspaces of the hyperplanes of ${\mathcal S}'$ containing $K'$. 

If the normal vectors of ${\mathcal H}'$ generate a proper subspace $W\subset\mathbb{R}^{d+1}$ of dimension $d'$, then we consider ${\mathcal H}'$ as an arrangement in $W$, intersect $K'$ with $W$, and affinely transform $W$ to $\mathbb{R}^{d'}$. This yields a realization $(\cH'',K'')$ of $\mathcal{M}$ with the desired properties.
\end{proof}

\begin{proposition}\label{orthants}
Let $\mathcal{M}=(U,\covectors)$ be a realizable COM. Then $\mathcal{M}$ has a realization $(\cH_n,K),$ where $K$ is a relatively open polyhedral cone, and $\cH_n$ is the set of the coordinates hyperplanes. 
\end{proposition}
\begin{proof} Let $(\cH'',K'')$ be a realization of $\cM$ in ${\mathbb R}^d$ provided by~\Cref{cones}. The vectors normal to the central hyperplanes of $\cH''$ generate ${\mathbb R}^d$. Now, we can extend $\cH''$ to a hyperplane arrangement in a larger ${\mathbb R}^{n}$, so that the normal vectors of $\cH'''$ are a basis of ${\mathbb R}^{n}$, while defining $K'''$ just as the restriction of $K''$ to ${\mathbb R}^d$. Now, there exists a linear transformation $f$  of  ${\mathbb R}^{n}$ mapping these vectors to the coordinate vectors. Then $f$ maps the hyperplanes of $\cH'''$ to the coordinate hyperplanes $\cH_n$ and the relatively open polyhedron $K'''$ to a relatively open polyhedron $K$. Furthermore, each relatively open cell $C'''$ of $K'''$ bounded by a set $\cH_0'''$ of hyperplanes of $\cH'''$ is mapped to a nonempty  cell $C$ of $K$ bounded by the set $\cH_0$ of hyperplanes of $\cH_n$, which are images of the hyperplanes of $\cH_0'''$.  Consequently, $(\cH_n,K)$ realizes $\cM$. 
\end{proof}

\subsection{General properties of realizability}\label{sec:general} 
We continue with some properties of realizable COMs and OMs. For this we continue with the notions of restriction,  contraction, and minors for COMs.
Let $\cM=(U,\covectors)$ be a COM and $A\subseteq U$. Given a sign vector $X\in\{\pm 1, 0\}^U$  by $X\setminus A$ we refer to the \emph{restriction}
of $X$ to $U\setminus A$, that is $X\setminus A\in\{\pm1, 0\}^{U\setminus A}$ with $(X\setminus A)_e=X_e$ for all $e\in U\setminus A$. Note that in the realizable setting $\cM(\cH,K)$ this operation corresponds to adding some hyperplanes from $\cH$ as halfspaces to $K$.
The \emph{deletion} of $A$ is defined as $(U\setminus A,\covectors\setminus A)$, where $\covectors\setminus A:=\{X\setminus A:  X\in\covectors\}$. Note that in the realizable setting $\cM(\cH,K)$ this operation corresponds to removing some hyperplanes from $\cH$.
The \emph{contraction} of $A$ is defined as $(U\setminus A,\covectors/ A)$, where
$\covectors/ A:=\{X\setminus A: X\in\covectors\text{ and }\underline{X}\cap A=\varnothing\}$. Note that in the realizable setting $\cM(\cH,K)$ this operation corresponds to intersecting $K$ with some hyperplanes from $\cH$. If $\covectors'$ arises by
deletions and contractions from $\covectors$, $\covectors'$ is said to be \emph{minor} of  $\covectors$. With the arguments about realizability together with~\cite{BaChKn} we get.

\begin{lemma}[\!\!{\cite[Lemma 1]{BaChKn}}]\label{lem:minorclosed}
  The class of realizable COMs is closed under taking minors and
  restrictions. Moreover, if $\cM'$ is obtained by such operations
  from $\cM$, then $\edim(\cM')\leq\edim(\cM)$.
\end{lemma}

A COM $\cM=(U,\covectors)$ is called \emph{free} of dimension $d$ if its tope graph is isomorphic to $\Q_d$. A COM $\cM$ has VC-dimension $d$ if the topes of $\cM$ define a set-family of VC-dimension $d$. 

\begin{lemma}[\!\!{\cite[Lemma 1]{ChKnPh}}]\label{VCdim_d}
  A COM has VC-dimension $\le d$ if and only if it does not have a
  free COM of dimension $d+1$ as a minor.
\end{lemma}

\begin{lemma}\label{dimfree}
    The Euclidean dimension of a free COM of dimension $d$ is at least $d$.
\end{lemma}
\begin{proof}
     A hyperplane arrangement of $d$ hyperplanes in $\mathbb{R}^{d-1}$ has less than $2^{d}$ maximal cells, see e.g.,~\cite[Exercise 4.3.5]{BjLVStWhZi}. Thus, the free COM  (which has $2^d$ topes) cannot be realized in $\mathbb{R}^{d-1}$. 
\end{proof}



The previous lemmas together imply:

\begin{theorem}\label{generalbounds}
For every realizable COM $\mathcal{M}$, we have $\VCdim(\mathcal{M})\leq \edim(\mathcal{M})\leq |U|$.
\end{theorem}
\begin{proof}
For the lower bound, if $\mathcal{M}$ has $\VCdim(\mathcal{M})=r$, then by~\Cref{VCdim_d} we can remove hyperplanes from its realization and obtain a realization of $\Q_r$, which is an OM of VC-dimension $r$ by~\Cref{dimfree}, hence by~\Cref{lem:minorclosed} it has  $r\leq \edim(\Q_r)\leq \edim(\mathcal{M})$. The upper bound is~\Cref{orthants}.
\end{proof}

As a side remark we provide a quick answer to a question communicated to us privately by Kunin, Lienkaemper, and Rosen, i.e., that a non-realizable OM cannot be realized as a COM either:
\begin{remark}
    If $\mathcal{M}$ is an OM realizable as a COM, then $\mathcal{M}$ is realizable by a central hyperplane arrangement.
\end{remark}
\begin{proof}
Let $\mathcal{M}=(U,\covectors)$ be an OM and suppose that it can be realized as a (realizable) COM $\mathcal{M}=\mathcal{M}(\cH,K)$ in $\mathbb{R}^d$. Suppose without loss of generality that $K$ is full-dimensional, otherwise we project into the affine hull of $K$. Since $\mathcal{M}$ is an OM, we have that  $\mathbf{0}\in \covectors$ (recall that $\mathbf{0}$ is the  all-zeros vector). Hence, there is a point $x\in K$ in which all the hyperplanes of $\cH$ intersect. After translating $x$ to the origin of $\mathbb{R}^d$ we see that $\cH$ is a central hyperplane arrangement. Hence, its combinatorics is determined by any arbitrary small sphere $S$ around $x$, see e.g.~\cite[Section 1.2.(c)]{BjLVStWhZi}. Therefore, we can set $K=\mathbb{R}^d$ and get a a representation of $\cM$ by a central hyperplane arrangement.
\end{proof}





\subsection{Generalized convex shellings} Convex geometries are
considered as generalizations of Euclidean convexity in the following
sense. Let $\cA = (U,\cC)$ be a convex geometry. We say that $\cA$ is
\emph{CG-realizable} if $U$ can be viewed as a finite subset of points
of ${\mathbb R}^d$ such that $X\subseteq U$ belongs to $\cC$ if and
only if the Euclidean convex hull of $X$ does not contain other points
of $U$: $\conv_{{\mathbb R}^d}(X)\cap U=X$. Not every convex geometry
can be realized in this way. Kashiwabara et al. \cite{KaNaOk} gave a
realization theorem for all convex geometries using the notion of
generalized convex shellings. Let $P$ and $Q$ be two finite sets of
${\mathbb R}^d$ such that the Euclidean convex hull of $Q$ does not
intersect $P$. Then the family of subsets
$\mathcal{C}(P,Q)=\{ X\subseteq P: \conv(X\cup Q)\cap P=X\}$ of $P$ is
called the \emph{generalized convex shelling on} $P$ \emph{with
  respect to}
$Q$.  
It is easy to see that any generalized convex shelling is a convex
geometry. Conversely, Kashiwabara et al. \cite{KaNaOk} proved that any
convex geometry arises as a generalized convex shelling. If
$Q=\varnothing$, then this leads to the notion of CG-realizability. If
${\cC}=\mathcal{C}(P,Q)$, then we will say that $\mathcal{C}(P,Q)$ is
a \emph{generalized convex shelling} of $\cA=(P,\cC)$. Richter and
Rogers \cite{RiRo} proved that any convex geometry of convex
dimension $d$ has a generalized convex shelling in ${\mathbb R}^d$.


Extending the notation for generalized convex shellings, given two finite point sets $P$ and $Q$ of ${\mathbb R}^d$,  let $$\cM(P,Q)=\{ X\subseteq P: \exists \mbox{ halfspace } H^+ \mbox{ s.t. } Q\subset H^+ \mbox{ and } P\cap H^+=X\}.$$ 
 
We say that $(P,Q)$ is a \emph{point representation} of a COM $\cM$ if $\M=\cM(P,Q)$, i.e., the set $\mathcal{T}$ of topes of $\cM$ consists  precisely of the $\{-1,+1\}$-vectors of all sets of $\cM(P,Q)$ as in~\Cref{setsystem}. Trivially, $P\in \cM(P,Q)$. Now, we establish a correspondence between (hyperplane) 
realizations and point realizations of COMs and a link with generalized convex shellings. For this purpose we define a COM to be \emph{acyclic} if it has the all plus-vector $(+1,\ldots,+1)$ as a tope. 



\begin{proposition}[Point representations of COMs] \label{point-config} An acyclic COM $\mathcal{M}=(U,\covectors)$ is (hyperplane) realizable if and only if $\cM$ is point realizable. 
\end{proposition}
\begin{proof}
Let $\mathcal{M}=(U,\covectors)$ be an acyclic realizable COM. We have $\mathcal{M}=\mathcal{M}(\cH',K')$ where ${\mathcal H}'$ is a central hyperplane arrangement ${\mathcal H}'$ and $K'$ an open polyhedral cone in $\mathbb{R}^{d+1}$ as in \Cref{cones}. Now, the vectors $v_e$, normal to the hyperplanes $H'_e, e\in U$ of the arrangement ${\mathcal H}'$, gives rise to a set ${\mathcal V}$ of vectors of ${\mathbb R}^{d+1}$. Analogously, the vectors $u_j$, normal to the hyperplanes $S'_j$ of ${\mathcal S}'$ defining the polyhedron $K'$ and such that their product with the vectors with ends inside $K'$ is positive
, gives rise to a set ${\mathcal U}$ of vectors of ${\mathbb R}^{d+1}$. For a cell $C$ of the initial realization, each point $p$ of $C$ defines a vector  $v_p$ belonging to the cone $C'$. Let $H_{v_p}$ be the central hyperplane of ${\mathbb R}^{d+1}$ having $v_p$ as the normal vector. Then all vectors of ${\mathcal U}$ belong to a halfspace defined by $H_{v_p}$. Analogously, each vector $v_e\in {\mathcal V}$ belongs to the positive side of the hyperplane $H_{v_p}$ if and only if the cell $C$ is located on the positive side of the hyperplane $H_e\in {\mathcal H}$ (and thus the cone $C'$ is on the positive side of the central hyperplane $H'_e\in {\mathcal H}$). 
Therefore, the topes of $\cM$ can be identified (in the set-theoretical language) with the subsets $X$ of $U$ such that there exists a central hyperplane $H'$ of  ${\mathbb R}^{d+1}$ such that all $U$ belong to one halfspace defined by $H'$ and a vector $v_e, e\in U$ belongs to the same halfspace defined by $H'$ if and only if $e\in X$. 

Now, we can return back to ${\mathbb R}^d$ by considering a hyperplane $H_0$ of ${\mathbb R}^{d+1}$ parallel to ${\mathbb R}^d$, which intersects all vectors of ${\mathcal V}\cup  {\mathcal U}$. This hyperplane exists, since by acyclicity of $\cM$ all elements of ${\mathcal V}$ have a positive 
last coordinate and the same holds for the elements of ${\mathcal U}$ by the choice of normal vectors for the defining hyperplanes of $K'$. Let $p_e\in H_0$ denote the point defined by the vector $v_e\in {\mathcal V}$ and let $q_j\in H_0$ denote the point defined by the vector $u_j\in {\mathcal U}$. Denote by $P$ and $Q$ the resulting point configurations of ${\mathbb R}^d$. The correspondence between points in the model $(\cH',K')$ and hyperplanes in the model $(P,Q)$ establishes that for any point $p$ in a cell $C$ of $(\cH',K')$, and thus for any tope $X$ of $\cM(\cH',K')$, there is a hyperplane $H_p$ such that $H_p$ induces the sign-pattern of $X$ in the model $(P,Q)$.

Conversely, the preceding transformation can be reversed. Namely, given sets $P,Q$ in $H_0\subseteq {\mathbb R}^{d+1}$, we can define the hyperplane arrangements $\cH$ and $\mathcal{S}$ having $P$ and $Q$ as normal vectors, respectively. Letting $K$ be the polyhedral cone defined by $\mathcal{S}$, we get a realization $\cM=\cM(\cH,K)$.
\end{proof}




\begin{remark} \label{point-rep-conv-shell}
  The essential difference between point representation of~\Cref{point-config} provided by ${\mathcal M}(P,Q)$ and generalized convex shelling of \cite{KaNaOk} provided by ${\mathcal C}(P,Q)$ is that instead of a single halfspace $H^{+}$ containing $Q$ and intersecting $P$ in $X$, the generalized convex shelling requires that $\conv(X\cup Q)\cap P=X$. This is equivalent to the existence of a set of halfspaces $H^+_1,\ldots,H^+_k$, all 
  containing $Q$ and such that $X=P\cap (\cap_{i=1}^k H_i^+)$. 
  (As a set of halfspaces $H^+_1,\ldots,H^+_k$ one can take the halfspaces defined by the support hyperplanes of the facets of the polytope $\conv(X\cup Q)$ and containing this polytope).  This is a natural requirement for intersection-closed set families, because any set family admitting a 
generalized convex shelling is closed by intersections: if $H_1^+,\ldots,H^+_k$ and $S^+_1,\ldots,S^+_m$ are two sets of halfspaces such that 
$Q\subset (\cap_{i=1}^k H^+_i)\cap (\cap_{j=1}^m S^+_j)$ and $P\cap (\cap_{i=1}^k H^+_i)=X$, $P\cap (\cap_{i=1}^{m} S^+_j)=Y$, then 
$P\cap (\cap_{i=1}^k H^+_i)\cap  (\cap_{i=1}^{m} S^+_j)=X\cap Y$. 

Therefore, the hyperplane realizations and point representations of COMs are much simpler than generalized convex shellings and correspond to classical OM-realizations. 
The main results of our paper (\Cref{realizations-convex-geometries} and \Cref{th-convgeom-cdim}) must
be considered from this point of view, while compared to the results
of \cite{KaNaOk} and \cite{RiRo} in the case of convex geometries. However, in contrast to our results generalized convex shellings characterize convex
geometries \cite{KaNaOk}.
\end{remark}

In the following we establish a connection between point representations and generalized convex shellings.
A COM $\cM$ is called \emph{intersection-closed} if the set of topes of $\M$ defines an intersection-closed family of sets. Note that by simplicity this implies that $(-1, \ldots, -1)$ is a tope of $\cM$. Obviously, convex geometries are acyclic and their ideals are intersection closed COMs. Acyclicity corresponds to the universe being convex, i.e., $(+1, \ldots, +1)$ is a tope of $\cM$.

\begin{proposition}\label{pointshelling}
Let $P,Q \subset \mathbb{R}^d$ be finite. If $\cM=\cM(P,Q)$ is a simple acyclic intersection-closed COM, then $\cC(P,Q)$ is a generalized convex shelling. In particular, $\cM$ is a convex geometry.
\end{proposition}
\begin{proof} 
Since $\cM=\cM(P,Q)$ contains the tope $(-1,\ldots,-1)$, necessarily  $\conv(Q)\cap P=\varnothing$. From the definition of the point representations and generalized convex shellings, we conclude that 
$\cM\subseteq \cC(P,Q)$. 
To prove the converse inclusion  $\cC(P,Q)\subseteq \cM$, pick any set $X$ from $\cC(P,Q)$. Then there exists a set $H^+_1,\ldots,H^+_k$ of halfspaces, all  containing $Q$ and such that $X=P\cap (\cap H_i^+)$. Let $X_i=P\cap H_i^+, i=1,\ldots,k$. Then each 
$X_i$ belongs to $\cM(P,Q)$. Since $\cM(P,Q)=\cM$, each $X_i$ belongs to $\cM$. Since $\cM$ is 
intersection-closed and $\cap_{i=1}^kX_i=X$, we conclude that $X\in \cM$, establishing that $\cC(P,Q)\subseteq \cM$. 
\end{proof}




\section{Convex geometries and their ideals} 
In this section, we present further results about convex geometries and their ideals that will be used in our proofs. We also present examples of ideals of convex geometries.

\subsection{Rooted circuits} In the proof of realizability of convex geometries, we will  use the characterization of convex geometries via rooted circuits and critical rooted circuits.  
Let $\cA= (U,\mathcal{C})$ be a convex geometry. 
A \emph{rooted set}  is a pair $(C,r)$ consisting of a subset $C$ of $U$ and an element $r$ of $C$. A convex geometry $\cC$ is \emph{reconstructed} from a collection of rooted sets $\cF$ if $\mathcal{C}=\{ X\subseteq U: (C,r)\in \cF\Rightarrow X\cap C\ne C\setminus \{ r\}\},$ i.e., $X\subseteq U$ belongs to $\cC$ if and only if no rooted set $(C,r)\in \cF$ 
meets $X$ in $C\setminus \{ r\}$. A rooted set $(C,r)$ is a \emph{rooted circuit} \cite{KoLo} of $\mathcal{C}$ if  $\mathcal{C}|_C=2^C\setminus \{ C\setminus \{ r\}\}$. Denote by $\cR(\cA)$ the set of all rooted circuits of $\mathcal A$. $\cR(\cA)$ has  the following properties established in \cite{KoLo}  (for proofs, see Lemma 2 and Proposition 4 of  \cite{Di}):

\begin{theorem}[\!\!\cite{KoLo, KoLoSch}]\label{rootedcircuits1}
  Let $\cA = (U,\mathcal{C})$ be a convex geometry. Then:
\begin{itemize}
\item[(i)] If $(C,r)$ is a rooted circuit of $\mathcal{C}$, then $r\in \conv(C\setminus\{r\})$;
\item[(ii)] $\cA$ can be reconstructed from the family $\cR(\cA)$ of its rooted circuits. 
\end{itemize} 
\end{theorem} 

Dietrich \cite{Di} provided the following axiomatization of convex geometries via the rooted circuits: 

\begin{theorem}[\!\!\cite{Di}]\label{dietrich}
Let $\mathcal R$ be a set of rooted subsets of a finite set $U$. Then
$\mathcal{R}$ is the set of rooted circuits of a convex geometry if
and only if $\mathcal{R}$ satisfies the following two properties:
\begin{itemize}
    \item[(1)] $(C_1,r_1),(C_2,r_2)\in \mathcal{R}$ and $C_1\subseteq C_2$ implies $C_1=C_2$ and $r_1=r_2$;
    \item[(2)] $(C_1,r_1),(C_2,r_2)\in \mathcal{R}$ and $r_1\in C_2\setminus \{ r_2\}$ implies that there exists $(C_3,r_2)\in \mathcal{R}$ such that $C_3\subseteq (C_1\cup C_2)\setminus \{r_1\}$.
\end{itemize}
\end{theorem} 

Korte and Lov\'asz \cite{KoLo} (see also the book \cite{KoLoSch})
identified a canonical subset of rooted circuits, which they call
critical rooted circuits. In terms of convex geometries, a rooted
circuit $(C,a)$ of a convex geometry $\cA = (U,\cC)$ is a
\emph{critical rooted circuit} if $\conv(C)\setminus \{ a\}$ does not
belong to $\cC$ but $\conv(C)\setminus \{ a,b\}$ belongs to $\cC$ for
any $b\in C\setminus \{ a\}$. Denote by $\mathcal{R}_0(\cA)$ the set
of all critical rooted circuits of $\cA$.  The importance of the
family $\cR_0(\cA)$ of critical rooted circuits stems from the fact
that $\cA$ can be uniquely determined by $\cR_0(\cA)$ and $\cR_0(\cA)$
is minimal with respect to this property. This is: if $\cA$ can be
determined by $\cR'\subseteq \cR(\cA)$, then
$\cR_0(\cA)\subseteq \cR'$ \cite{KoLo, KoLoSch} (for a precise
definition of ``determined by'', see \cite{KoLoSch}).


To each rooted set $(C,r)$ of a convex geometry $\cA=(U,\cC)$ we can associate the cube $Q(C,r)$: $Q(C,r)$ consists of all 
subsets of $U$ of the form $(C\setminus{r})\cup X'$ with $X'\subseteq U\setminus C$. The cube $Q(C,r)$ can be encoded by the sign vector $X(C,r)$, where $X_e(C,r)=+1$ if $e\in C\setminus \{ r\}$, $X_r(C,r)=-1$, and $X_e(C,r)=0$ if $e\in U\setminus C$. 
If $(C,r)$ is a rooted circuit, then the trace of $\cC$ on $C$ coincides with $2^C\setminus \{ C\setminus \{ r\}\}$, thus the cube $Q(C,r)$ belongs to the complement $\cA^*$ of $\cA$ and, furthermore,  $Q(C,r)$ is a maximal by inclusion cube of $\cA^*$.  
Consequently,  $Q(C,r)$ is an $(U\setminus C)$-cube of 
$\cA^*$. The converse also holds:

\begin{lemma} \label{rooted-circuits-cubes} Let $\cA=(U,\cC)$ be a convex geometry. If $Q$ is a maximal cube of $\cA^*$, then there exists $(C,r)\in \cR(\cA)$ such that $Q=Q(C,r)$. 
\end{lemma}

\begin{proof} Suppose that $Q$ consists of the sets $\{ Y\cup X': X'\subseteq X\}$ for $Y\subseteq U\setminus X$. Since $Y\notin \cC$, by \Cref{rootedcircuits1}(ii), there exists a rooted circuit $(C,r)$ such that $Y\cap C=C\setminus \{ r\}.$ But then $Q$ is contained in the cube $Q(C,r)$ of $\cA^*$ whose vertices are the sets $\{ (C\setminus \{ r\})\cup X': X'\subseteq U\setminus C\}$. Since $Q$ is maximal, $Q=Q(C,r)$. 
\end{proof}

\subsection{Positive circuits}
In case of ideals of convex geometries, along with rooted circuits we
also have to define positive circuits.  Let $\cI=(U',\cC')$ be an ideal of a
convex geometry $\cA=(U,\cC)$. A \emph{positive circuit} of $\cI$ \emph{with
  respect to} $\cA$ is a subset $P \subseteq U$ such that
$\cC'_{|P} = 2^P \setminus \{P\}$. In particular, for any
$e \in U\setminus U'$, $\{e\}$ is a positive circuit of $\cI$ with respect to
$\cA$.  Given an ideal $\cI$ of a convex geometry $\cA$, we denote by
$\cP(\cI)$ the set of positive circuits of $\cI$  with respect to $\cA$. We now show
that $\cC'$ is precisely the set of elements of
$\cC$ not containing positive circuits.

\begin{lemma}\label{lem-conflicts-def}
  For any ideal $\cI=(U',\cC')$ of a convex geometry $\cA=(U,\cC)$,
  we have
  \[\cC' = \{X \in \cC : \forall P \in \cP(\cI),  P\not\subseteq X\}.\]
\end{lemma}

\begin{proof}
  Let
  $\cC'' = \{X \in \cC : \forall P \in \cP(\cI), P\not\subseteq X\}$.
  Pick any $X \in \cC'$ and any $P \subseteq X$. Then
  $P \cap X = P \in \cC'_{|P}$ and thus
  $\cC'_{|P} \neq 2^P \setminus \{P\}$. This establishes that
  $P \notin \cP(\cI)$ for any $P \subseteq X$ and consequently
  $X \in \cC''$. Therefore $\cC' \subseteq \cC''$.

  Suppose now that there exists $Y \in \cC''\setminus \cC'$ and assume
  that $Y$ is such a set of minimal size. By minimality of $Y$, for
  any $e \in Y$, $Y\setminus \{e\} \in \cC''$ if and only if
  $Y \setminus \{e\} \in \cC'$. Let
  $P_0 = \{e \in Y : Y\setminus \{e\} \in \cC'\}$. Since $\cC'$ is
  intersection-closed, for any
  $\varnothing \subsetneq P' \subseteq P_0$, we have
  $Y \setminus P' \in \cC'$. This shows that
  $2^{P_0} \setminus \{P_0\} \subseteq \cC'_{|P_0} \subseteq 2^{P_0}$.
  Since $P_0 \subseteq Y$ and since $Y \in \cC''$, by the definition
  of $\cC''$, we get $P_0 \notin \cP(\cI)$ and thus
  $\cC'_{|P_0} = 2^{P_0}$ by the definition of $\cP(\cI)$.
  Consequently, there exists $X \in \cC'$ such that $P_0 \subseteq X$.
  Let $Z = X \cap Y$ and observe that $P_0 \subseteq Z$.  Since $\cC$
  is intersection-closed, $Z \in \cC$. If $Z = Y$, then
  $Y \subseteq X \in \cC'$, contradicting the fact that $\cC'$ is an
  ideal of $\cC$. Thus we have $Z \subsetneq Y$.  Since $Z \in \cC'$
  and $Y \notin \cC'$, there exists a set $Z \subseteq A \subseteq Y$
  and an element $e \in Y \setminus A$ such that $A \in \cC'$ and
  $B = A\cup \{e\} \in \cC\setminus \cC'$ ($A$ and $e$ exist since
  $\cC$ is ample and thus isometric). By minimality of $Y$, $B =
  Y$. Consequently $e \in P_0$ by the definition of $P_0$, but this is
  impossible since $P_0 \subseteq Z \subseteq A$.  This establishes
  that $\cC'' \subseteq \cC'$.
\end{proof}



\subsection{Convex dimension} Edelman and Jamison~\cite{EdJa} provided
a nice characterization of convex geometries via order convexity.
Given a universe $U = \{e_1,\ldots,e_n\}$, a total order $<$ on $U$,
and the reflexive closure $\leq$ of $<$, call an \emph{ending
  interval} of $\leq$ any set of the form $\{e\in U: e_i\leq e\}$ for
some $e_i\in U$.  Given a set $\Upsilon=\{ \leq_1,\ldots,\leq_d\}$ of
$d$ total orders on $U$, we say that a set family $\cC \subseteq 2^U$
is \emph{generated} by the set $\Upsilon$ if $\varnothing \in \cC$ and
a nonempty set $C$ belongs to $\cC$ if and only if $C$ is the
intersection of $d$ ending intervals, one from each order $\leq_i$ of
$\Upsilon$, i.e., if there exists $(e_i)_{1\leq i \leq d} \in U^d$
such that $C = \{e \in U : e_i \leq_i e, \forall 1 \leq i \leq d
\}$. Edelman and Jamison proved the following result:

\begin{theorem}[\!\!{\cite[Theorem 5.2]{EdJa}}] Any set family $\cC \subseteq 2^U$ 
generated by a set of total orders on $U$ is a convex geometry and, conversely,
any convex geometry $\cA=(U,\cC)$  can be generated by a set of total
orders on $U$. 
\end{theorem} 

The \emph{convex dimension} $\cdim(\cA)$ of a convex geometry $\cA=(U,\cC)$ is the least number of total orders generating $\cC$. 

\subsection{(Ideals of) convex geometries as meet-(semi)lattices} 
A poset $L=(X,\leq)$ is a \emph{meet-semilattice} if for every $x,y\in L$ there is unique largest element $x\wedge y\in L$ such that $x
\wedge y\leq x,y$;  $x\wedge y$ is called the \emph{meet} of $x$ and $y$.
A poset $L=(X,\leq)$ is a \emph{join-semilattice} if for every $x,y\in L$ there is unique smallest element $x\vee y\in L$ such that $x\vee y\geq x,y$;   $x\vee y$ is called the \emph{join} of $x$ and $y$. A poset is a \emph{lattice} if it is both a join- and a meet-semilattice. Note that both join and meet are associative and commutative operations, so in order to take the join or meet over all the elements of a set $Y$ we sometimes just write $\bigvee Y$ and $\bigwedge Y$, respectively. Note that if a poset has a global minimum $\hat{0}$ or a global maximum $\hat{1}$, then we set $\bigvee \emptyset=\hat{0}$ and $\bigwedge \emptyset=\hat{1}$, respectively. An element $x\in L$ is called \emph{join-irreducible} if $x=\bigvee Y$ implies $x\in Y$ for all $Y\subseteq L$. The set of join-irreducible elements is denoted by $\mathcal{J}(L)$. Similarly, an element $x\in L$ is called \emph{meet-irreducible} if $x=\bigwedge Y$ implies $x\in Y$ for all $Y\subseteq L$. The set of meet-irreducible elements is denoted $\mathcal{M}(L)$. Following Dilworth a \emph{lower locally distributive lattice} (LLD) is a lattice $L$ such that for all $x\in L$ there is a unique inclusion-minimal $J_x\subseteq \mathcal{J}(L)$ such that $x=\bigvee  J_x$ (in fact, 
Dilworth introduced the dual lattices, the \emph{upper locally distributive lattices (ULD)}). This definition is equivalent to Edelman's notion of semi-distributive lattice, but since there are several notions of semidistributivity in the literature, we prefer the name LLD. We refer to~\cite{Stern,Mon,KDiss} for different equivalent characterizations, different names, as well as many different instances of LLDs and ULDs. We continue with the characterization of lattices of convex geometries provided by Edelman~\cite{Edel}: 

\begin{theorem}[\!\!{\cite[Theorem 3.3]{Edel}}] \label{lattice-convexgeometry} A lattice is lower locally distributive if and only if it is isomorphic to the inclusion-order of convex sets of a convex geometry.
\end{theorem}

The meet-irreducible elements of $L=(\mathcal{C},\subseteq)$ correspond to the copoints of the convex geometry  
$\cA=(U,\mathcal{C})$; a \emph{copoint} attached at point $p$ is a maximal by inclusion convex set of $\mathcal{C}$ not containing $p$. Edelman and Jamison \cite{EdJa} characterized convex geometries as the convexity spaces for which each copoint has a unique attaching point. Furthermore, Edelman and Saks \cite{EdSa}  characterized the convex dimension $\cdim(\cC)$ of a convex geometry $\cA$ in term of the poset of all meet-irreducibles  in the following nice way: 

\begin{theorem}[\!\!\cite{EdSa}] \label{cdim-convexgeom} For a convex geometry $\cA=(U,\mathcal{C})$, $\cdim(\cC)$ is equal to the size of the largest antichain of the poset $(\mathcal{M}(L),\subseteq)$, where $L=(\mathcal{C},\subseteq)$. 
\end{theorem}

An \emph{ideal} $I$ of a poset is a subset that is downwards-closed, i.e., if $x\leq y$ and $y\in I$, then $x\in I$. From~\Cref{lattice-convexgeometry} by definition we get a justification of the name \emph{ideals} of convex geometries:
\begin{remark}
    A meet-semilattice is an ideal of a lower locally distributive lattice if and only if it is isomorphic to the inclusion-order of convex sets of an ideal of a convex geometry.
\end{remark}



\subsection{Examples of ideals of convex geometries}
Edelman and Jamison~\cite{EdJa} and Korte et al.~\cite{KoLoSch}
presented numerous examples of convex geometries and antimatroids,
arising from geometry, language theory, and chip firing games. Bandelt
et al.~\cite{BaChDrKo} presented simplicial complexes and median set
systems as examples of bouquets of convex geometries.  We show that
multisimplicial complexes and median set systems are ideals of
particular convex geometries, called downset alignments. 
A \emph{simplicial complex} $L$ on a set $U$ is a family of subsets of $U$, called \emph{simplices} or \emph{faces} of $L$,
such that if $\sigma\in L$ and $\sigma'\subseteq \sigma$, then $\sigma'\in L$. The \emph{facets} of $L$ are the maximal (by inclusion)
faces of $L$. The \emph{dimension} $d$ of $L$ is the size of its largest face minus one. A multi-subset $\sigma$ of $U$ is a subset of $U$ such that
each element $e\in U$ is given with its multiplicity $n_{\sigma}(e)$ (the number of times, $e$ occurs in $\sigma$). A \emph{multisimplicial complex} $L$ on a set $U$ is a family of multi-subsets of $U$, such that if $\sigma\in L$ and $\sigma'$ is a multi-subset of $U$ such that
$n_{\sigma'}(e)\le n_{\sigma}(e)$ for each $e\in U$, then $\sigma'\in L$. The \emph{size} of a face $\sigma$ of $L$
is the size of $\{ e\in U: n_{\sigma}(e)>0\}$. The \emph{dimension} $d$ of $L$ is the size of its largest face minus one.


A \emph{median set system} is a set family 
$(U,\cC)$ satisfying (C1$'$), (C2), and 
\begin{itemize}
\item[(C6)] for any $x \neq y$ in $U$, there exists some
$K \in \cC$ with $|\{ x,y\} \cap K| =1$,
\item[(C7)] $K_i, M_i \in \cC$ $(i=1,2,3)$ with $K_i \cup K_j
\subseteq M_k$ for $\{i,j,k\} = \{1,2,3\}$ implies $K_1 \cup K_2 \cup K_3 \in \cC$.
\end{itemize}
The name is justified by the fact that by virtue of (C2) and (C7), $\cC$ is closed under the median operation
$m$  of $2^U$ defined by
$$m(L_1, L_2, L_3) := (L_1 \cap L_2) \cup (L_1 \cap L_3) \cup (L_2 \cap L_3).$$
The resulting meet-semilattice is called a \emph{median semilattice} \cite{Av,BiKi,Shol}. Median semilattices are locally lower distributive. Indeed, by a result of Birkhoff and Kiss~\cite{BiKi} they can be characterized by the  property that all 
principal ideals $\downarr x$ are distributive lattices and three elements have an upper bound whenever each pair of them does. 

Every abstract finite median algebra (for which the former set-theoretic
ternary operation is axiomatized) or, equivalently, any median graph can be represented 
by a median set system via the Sholander embedding~\cite{Shol} into some power set $2^U$. 
An inherent feature of median algebras/graphs is that they may be oriented so that
any element can serve as the empty set in the associated set representation: a
median set system $\cC$ is mapped onto another one $\cC \vartriangle Z := \{A \vartriangle Z: K \in \cC \},$
by the automorphism of $2^U$ taking the symmetric difference with a fixed
set $Z \in \cC$. 


Let $\leq$ be a partial order on $U$. The \emph{downset alignment} ${\mathcal D}_P$ \cite{EdJa} of the poset $P=(U,\leq)$ consists of all ideals of $P$. It was noticed in \cite{EdJa} that the downset alignments are convex geometries and it was shown in \cite[Theorem 3.2]{EdJa} that a convex geometry $\cA=(U,\mathcal{C})$ is a downset alignment if and only if $\mathcal{C}$ is union-closed. From the lattice point of view, downset alignments simply correspond to distributive lattices. 

We now define bouquets of downset alignments in the same vein as
bouquets of convex geometries were defined.  We say that a pair
$\cA=(U,\cC)$ is a \emph{bouquet of downset alignments} if $\mathcal{C}$
satisfies (C1$'$), (C2), (C3$'$), and
\begin{itemize}
\item[(C8)] for all $X,Y\in \mathcal{C}$, if there exists $Z \in \mathcal{C}$ such that  $X \cup Y \subseteq Z$, then $X\cup Y \in \mathcal{C}$. \hfill (locally union-closed)
\end{itemize}

Note that bouquets of downset alignments are precisely bouquets of
convex geometries with the additional property of being locally
union-closed. Note also that principal ideals of bouquets of downset alignments
are downset alignments and that ideals of downset alignments are bouquets
of downset alignments.  Differently from bouquets of convex geometries,
that are not always ideals of convex geometries by~\Cref{nonideal}, the next result shows that bouquets of
downset alignments and ideals of downset alignments
coincide:



\begin{theorem}
    \label{ideal}
Every bouquet of downset alignments $\cA=(U,\mathcal{C})$ with $\ell$ maximal convex sets is an ideal of a downset alignment $\cA^+$ such that $\vcd(\cA^+)\leq \ell\cdot\vcd(\cC)$. 
\end{theorem}
\begin{proof}
  Consider a bouquet of downset alignments $\cA=(U,\mathcal{C})$. Let
  $X_1, \ldots, X_\ell$ be the maximal convex sets of $\cA$ and let $\downarr X_1,\ldots,\downarr X_\ell$ be their principal ideals. For each
  $1 \leq i \leq \ell$, let $d_i= \vcd(\downarr X_i)$. We show by
  induction on $\ell$ that $\cA$ is an ideal of a downset alignment
  $\cA^+$ such that $\vcd(\cA^+)\leq \sum_{i=1}^\ell d_i$. 
 
  If $\ell = 1$, then $\cA$ is a downset alignment and we are
  done. Otherwise, take an element $M\in \cC$ that can be written as
  the intersection of at least two maximal convex sets of $\cA$, and
  that is inclusion maximal for this property. Let $\mathcal{S}$ be
  the set of all maximal convex sets of $\cA$ containing $M$, and
  assume without loss of generality, that
  $\mathcal{S}=\{X_1,\ldots,X_k\}$.  Consequently,
  $M = \bigcap_{i=1}^k X_i$.  Let $\cC'_0$ be the set of all sets of
  the form $Y_1\cup \ldots \cup Y_k$ where for each $0 \leq i \leq k$,
  $Y_i \in \cC$ and $M \subseteq Y_i \subseteq X_i$.  Now define
  $\cC'$ as $\cC \cup \cC'_0$ and set $\cA'=(U,\cC')$. Observe that
  the maximal convex sets of $\cA'$ are the sets
  $X_{k+1}, \ldots, X_{\ell}$, and $X'=\bigcup_{i=1}^kX_i$. 

  We first show that $\cC'$ is still a bouquet of downset alignments,
  i.e., that $\cC'$ satisfies (C1$'$), (C2), (C3$'$) and (C8).
  Observe that $\varnothing \in \cC \subseteq \cC'$, establishing
  (C1$'$).  To prove that $\cC'$ is intersection-closed, let
  $A,B\in \cC'$. If $A,B\in \cC$, then there is nothing to show, so
  suppose first $A\in \mathcal{C}$ and $B\in \cC'\setminus \cC$. Hence
  we can represent $B=Y_1\cup\ldots\cup Y_k$ and we consider
  $A\cap (Y_1\cup\ldots\cup Y_k)=(A\cap Y_1)\cup\ldots\cup (A\cap
  Y_k)$. Since all of these sets are subsets of $A$, since
  $\mathcal{C}$ is locally union-closed, their union is in
  $\mathcal{C}$.  If $A,B\in \cC'\setminus \cC$, then
  $A=Y_1\cup\ldots\cup Y_k$ and $B=Y'_1\cup\ldots\cup Y'_k$ and
  $(Y_1\cup\ldots\cup Y_k)\cap (Y'_1\cup\ldots\cup Y'_k)=Y_1\cap
  (Y'_1\cup\ldots\cup Y'_k) \cup \ldots Y_k\cap (Y'_1\cup\ldots\cup
  Y'_k)$ but for each $i$ here we have
  $M\subseteq Y_i\cap (Y'_1\cup\ldots\cup Y'_k)\subseteq X_i$, hence
  their union is in $\cC'$ by definition of $\cC'$.

  In order to show that $\cC'$ is locally union-closed let
  $A,B\in \cC'$ and let $A,B\subseteq Z$. Since we can assume that
  $A\notin\mathcal{C}$, without loss of generality we can set
  $Z=\bigcup_{i=1}^kX_i$ and $A=Y_1\cup\ldots\cup Y_k$ as before. Now,
  $B=(B\cap X_1)\cup\ldots\cup(B\cap X_k)$, where each term is in
  $\mathcal{C}$. Then,
  $A\cup B=(Y_1\cup (B\cap X_1)\cup\ldots\cup Y_k\cup (B\cap X_k))$,
  where each term is in $\mathcal{C}$ by local union-closedness and each term contains $M$ as a subset. Hence
  by definition of $\cC'$, $A\cup B\in\cC'$.

  Now, let us show that $\cC'$ is locally extendable. For this purpose
  let $A,B\in \cC'$ be distinct sets such that $B$ is a maximal convex
  set and $A\subseteq B$. If $B\in \mathcal{C}$, then since
  $\downarr B$ does not contains any set of $\cC'_0$, also
  $A\in \mathcal{C}$ and there is nothing to show. So suppose that
  $B\in \cC'\setminus \mathcal{C}$, i.e., that
  $B=X'=\bigcup_{i=1}^kX_i$ (since $B$ is a maximal convex set of
  $\cC'$). If $A\in \mathcal{C}$, then $A\cap X_i\subsetneq X_i$ for
  some $1 \leq i \leq k$, and thus, there is an element
  $e\in X_i\setminus A\subseteq B\setminus A$ such that
  $(A\cap X_i)\cup e\in \mathcal{C}$ since $\mathcal{C}$ is locally
  extendable. Now, $\cC'$ is locally union-closed also
  $A\cup(A\cap X_i)\cup e=A\cup e\in\cC'$.  If
  $A\in\cC'\setminus \cC$, then $A=Y_1\cup\ldots\cup Y_k$ with
  $Y_i \subseteq X_i$ for $1 \leq i \leq k$ and
  $Y_{j}\subsetneq X_{j}$ for some $1 \leq j \leq k$ (since
  $A \subsetneq B = X'$). Since $\mathcal{C}$ is locally extendable,
  there is $e\in X_j\setminus Y_j\subseteq B\setminus A$ such that
  $Y_j\cup e\in \mathcal{C}$. But then by definition of $\cC'$, we
  have that $A\cup e\in \cC'$.

  Let us now show that $\cC$ is an ideal of $\cC'$. For this, suppose
  there exist $X \in \cC$ and $R \in \cC'_0 = \cC' \setminus \cC$ such
  that $R \subseteq X$. We can assume that $X$ is a maximal convex set
  and since $M \subseteq R$ by definition of $\cC'_0$, $X$ necessarily
  belongs to $\mathcal{S}$ and we can assume that $X = X_1$. Let
  $R=Y_1\cup \ldots \cup Y_k\subseteq X_1$ where for all
  $1\leq i\leq k$, $M \subseteq Y_i\subseteq X_i$ and $Y_i \in
  C$. Since $R \notin \cC$, necessarily $Y_1 \subsetneq R$ and thus we
  can assume that $M \subsetneq Y_2$. Since,
  $ Y_2 \subseteq R \subseteq X_1$ and
  $Y_2 \subseteq X_2$, we have 
  $M \subsetneq Y_2 \subseteq X_1 \cap X_2$, contradicting the
  maximality of $M$.

  We now show that
  $\vcd(\downarr X') \leq \sum_{i={1}}^k \vcd(\downarr X_i)$.
  Consider a subset $Z \subseteq X'$ with $|Z| = \vcd(\downarr X')$
  that is shattered by $\downarr X'$. For each $1 \leq i \leq k$, let
  $Z_i = X \cap X_i$. We claim that $Z_i$ is shattered by
  $\downarr X_i$. Indeed, for any $Z'' \subseteq Z_i \subseteq Z$,
  there exists $X'' \in \downarr X'$ such that $X'' \cap Z =
  Z''$. Observe that
  $(X'' \cap X_i) \cap Z_i = X'' \cap X_i \cap Z = Z''$ since
  $X'' \cap Z = Z''\subseteq Z_i \subseteq X_i$. Note also that
  $X'' \cap X_i \in \cC'$ and thus $X'' \cap X_i \in \downarr
  X_i$. Consequently, $Z_i$ is shattered by $\downarr X_i$. Therefore,
  since $Z = \cup_{i=1}^k Z_i$,
  $\vcd(\downarr X') = |Z| \leq \sum_{i=1}^k |Z_i| \leq \sum_{i=1}^k
  \vcd(\downarr X_i)$, and we are done.

  Consequently, $\cA'$ is a bouquet of downset alignments, $\cA$ is an
  ideal of $\cA'$, and the maximal convex sets of $\cA'$ are the sets
  $X_{k+1}, \ldots, X_{\ell}$, and $X'=\bigcup_{i=1}^kX_i$. Since
  $k \geq 2$, $\cC'$ has less inclusion-maximal convex sets than
  $\cC$. By induction hypothesis, $\cA'$ is an ideal of a downset
  alignment $\cA^+$ such that
  $\vcd(\cA^+) \leq \vcd(\downarr X') + \sum_{i={k+1}}^\ell
  \vcd(\downarr X_i) \leq \sum_{i={1}}^k \vcd(\downarr X_i) +
  \sum_{i={k+1}}^\ell \vcd(\downarr X_i) = \sum_{i={1}}^\ell
  \vcd(\downarr X_i)$. By transitivity, $\cA$ is an ideal of
  $\cA^+$. This concludes the proof.
\end{proof}

The grid $\mathbb{N}^n$ can be viewed as the covering graph of the poset $(\mathbb{N}^n,\leq)$, where for $x,y\in \mathbb{N}^n$ with $x=(x_1,\ldots,x_n)$ and $y=(y_1,\ldots, y_n)$, we set $x\leq y$ if and only $x_i\le y_i$ for $i=1,\ldots,n$. Since ideals of the grid $(\mathbb{N}^n,\leq)$ are ample, the VC-dimension $\VCdim(L)$ of any such ideal $L$ is the largest
subcube of $L$. Notice that $\VCdim(L)$ can be smaller than the dimension $n$ of the grid.
It is obvious that there exists a bijection between the ideals of the grid $(\mathbb{N}^n,\leq)$ and the multisimplicial complexes on a set $X$ of size $n$. The simplicial complexes are in bijection with the ideals of the $n$-cube $\{0,1\}^n\subset \mathbb{N}^n$. The VC-dimension of any (multi)simplicial complex coincides with its dimension. Hence, we can apply \Cref{ideal} to obtain the following result:

\begin{corollary}\label{co:ideal}
 Every multisimplicial complex of dimension $d$ and $\ell$ facets is an  ideal of a downset alignment of VC-dimension at most $\ell d$.
\end{corollary}

We get an analogous bound for median set systems (since they principal ideals are distributive lattices, median set systems are bouquets of downset alignments):

\begin{corollary}\label{co:mediansetsystem}
 Every median set system of VC-dimension $d$ and having $\ell$ maximal elements is an  ideal of a downset alignment of VC-dimension at most $\ell d$.   
\end{corollary}

\section{Realization of convex geometries and of their ideals}
In  this section we prove the main results of the paper: any ideal $\cI$ of  $\cA$ of a convex geometry $\cA=(U,\cC)$ admits a realization $(\cH_n,K(\cR_0))$ in ${\mathbb R}^n$, where $n=|U|$, $\cH_n$ are the coordinate hyperplanes of ${\mathbb R}^n$, and the number of facets of  $K(\cR_0)$ is at most the number of critical rooted circuits of $\cA$ and positive circuits of $\cI$. In particular, this yields a representation for convex geometries.

Furthermore, we prove that  each convex geometry is realizable in dimension equal to its convex dimension. Finally, we prove that trees and multisimplicial complexes have realizations in dimension bounded linearly in their VC-dimension. 

\subsection{Realization of ideals of convex geometries}

Let $U=\{ 1,\ldots,n\}$ and let $\cA=(U,\cC)$ be a convex geometry on
$U$ having $\cR:=\cR(\cA)$ as the set of rooted circuits
and $\cR_0:=\cR_0(\cA)$ as 
the set of critical rooted circuits.  For each rooted circuit $(C,r)$ of
$\cA$ consider the hyperplane $H(C,r)$ defined by the equation
$n x_r=\sum_{e\in C\setminus \{ r\}}
x_e$. 
Denote by $K(C,r)$ the open halfspace of ${\mathbb R}^n$ determined by
$H(C,r)$ as
\[K(C,r)=\left\{ x=(x_{1},\ldots,x_{n})\in {\mathbb R}^n: n x_r>\sum_{e\in C\setminus \{ r\}} x_e \right\}.\] 
Let $K(\cR)=\bigcap_{(C,r)\in \cR} K(C,r)$ and $K(\cR_0)=\bigcap_{(C,r)\in \cR_0} K(C,r)$ be the open polyhedra  that are respectively the intersection of all $K(C,r)$ taken over all rooted circuits $(C,r)$ of $\cA$ and the  intersection of all $K(C,r)$ taken over all critical rooted circuits $(C,r)$ of $\cA$. 
Obviously, $K(\cR)\subseteq K(\cR_0)$.

Consider now an ideal $\cI=(U',\cC')$ of the convex geometry
$\cA=(U,\mathcal{C})$. Let $\cP:=\cP(\cI)$ be the set of positive circuits of $\cI$
with respect to $\cA$.  For each $P \in \cP$, the hyperplane $H(P)$ is defined by the equation
$\sum_{e\in P} x_e=0$. Denote by $K(P)$ the open halfspace of ${\mathbb R}^n$ determined by
$H(P)$ as
\[K(P)=\left\{ x=(x_{1},\ldots,x_{n})\in {\mathbb R}^n: \sum_{e\in P} x_e<0\right\}.\] 

Finally, consider the open polyhedra $K(\cP) = \bigcap_{P\in \cP} K(P)$, $K= K(\cR) \cap K(\cP)$, and
$K_0=K(\cR_0) \cap K(\cP)$.  Clearly, $K \subseteq K_0\subseteq K(\cP)$.


\begin{theorem}\label{realizations-convex-geometries}
  For a convex geometry $\cA=(U,\cC)$, the pairs $(\cH_n,K(\cR))$ and
  $(\cH_n,K(\cR_0))$ are realizations of $\cA$. More generally, for any ideal $\cI=(U',\cC')$ of $\cA$, the pairs
  $(\cH_n,K)$ and $(\cH_n,K_0)$ are realizations of $\cI$.
\end{theorem}

The proof of the first assertion is a direct consequence of the second assertion. 
The proof that $(\cH_n,K(\cR))$ and $(\cH_n,K)$ are realizations of $\cA$
and $\cI$ is inspired by (but is simpler than) the proof of
Kashiwabara et al. \cite{KaNaOk} that convex geometries can be
represented via generalized convex shellings. The proof of the second 
assertion follows from four lemmas, which we prove first. 
We first show that for any $X \in \cC'$, the convex set $K$
intersects the $X$-orthant $\cO(X)$ of ${\mathbb R}^n$. 

\begin{lemma}\label{lem-Crealized}
  For any $X \in \cC'$, 
  the $X$-orthant $\cO(X)$ of ${\mathbb R}^n$ 
  intersects the convex polyhedron $K$. 
\end{lemma}

\begin{proof}
  We prove that the $X$-orthant $\cO(X)$ intersects the convex set $K$, i.e., 
  that there exists  a point
  $x^*=(x^*_{1},\ldots,x^*_{n})\in \cO(X)\cap K$. 
  Let $|X|=k$ and suppose without loss of generality
  that $X=\{ e_1,\ldots,e_k\}$.  First, we set
  $x^*_{e_1}=x^*_{e_2}=\cdots=x^*_{e_k}=1$. In the following, we set a
  negative value for each $x^*_e, e \notin X$. Observe that for any
  rooted circuit $(C,r)$ such that $C \subseteq X$, we have
  $n x^*_r = n = |U| > (|C|-1) = \sum_{e\in C\setminus \{ r\}} x_e$ and
  thus $x^* \in K(C,r)$, independently of the values of the remaining
  $n-k$ coordinates of $x^*$. Notice also that, since $X$ belongs to
  $\cC'$, $X$ does not contain any positive circuit of $\cC'$.
  Since $X\in \cC$, by axiom (C3) of convex geometries, the elements
  of $U\setminus X$ can be ordered $e_{k+1},\ldots,e_n$, such that all
  the sets
  $X_i=X\cup \{ x_{k+1},\ldots, x_{k+i}\}=X_{i-1}\cup \{ x_{k+i}\},
  i=1,\ldots,n-k$ are convex sets of $\cC$. We define the remaining
  $n-k$ coordinates of $x^*$ (which have to be negative) following the
  order $e_{k+1},\ldots,e_n$.  Suppose that after $i-1$ steps the
  coordinates $x^*_{e_{k+1}},\ldots,x^*_{e_{k+i-1}}$ have been
  defined. At stage $i$ we need to define the value of $x^*_{e_{k+i}}$
  in order to satisfy
  \begin{itemize}
  \item all constraints of the form
    $n x^*_r>\sum_{e\in C\setminus \{ r\}} x^*_e $, where
    $(C,r)\in \cR(\cA)$, $e_{k+i}\in C$, and $C\subset X_{k+i}$;
  \item all constraints of the form $\sum_{e\in P} x^*_e<0$, where
    $P \in \cP$, $e_{k+i}\in P$, and $P\subset X_{k+i}$.
  \end{itemize}

  For any rooted circuit $(C,r)$ such that
  $e_{k+i} \in C \subseteq X_{k+i}$, the values of $x^*_e$ for all
  elements $e\in C\setminus \{ e_{k+i}\}$ have been already
  defined. Since $X_{k+i-1} \in \cC$ and since
  $(C,r)$ is a rooted circuit of $\cA$, we know that $r \ne
  e_{k+i}$. Consequently, if we set
  $x^*_{e_{k+i}} < n x^*_r - \sum_{e\in C\setminus \{ r, e_{k+i}\}}
  x^*_e$, the point $x^*$ is in the halfspace $K(C,r)$. 

  Similarly, for any positive circuit $P$ of $\cI$  such that
  $e_{k+i} \in P \subseteq X_{k+i}$, the values of $x^*_e$ for all
  elements $e\in P\setminus \{ e_{k+i}\}$ have been already
  defined. Setting
  $x^*_{e_{k+i}} < - \sum_{e\in P\setminus \{ e_{k+i}\}} x^*_e$ ensures
  that the point $x^*$ is in the halfspace $K(P)$. By choosing a
  negative value for $x^*_{e_{k+i}}$ that is small enough, we can
  satisfy all the constraints over the rooted circuits $(C,r)$ and the
  positive circuits $P$ contained in the set $X_{k+i}$.

  Consequently, 
  we have constructed a point $x^*$ belonging to the $X$-orthant
  $\cO(X)$ and to the convex polyhedron~$K$.
\end{proof}

We now show that if a set $Y$ is not in $\cC$, then we can find a rooted circuit
$(C,r)$ whose hyperplane $H(C,r)$ separates the $Y$-orthant $\cO(Y)$ from the polyhedron 
$K(\cR)$ (and thus from $K$ and $K_0$). 

\begin{lemma}\label{lem-onlyCrealizedK}
  For any $Y \in \cC^*$, there exists a rooted circuit $(C,r)$ of $\cA$ such
  that $\cO(Y) \cap K(C,r) = \varnothing$. Consequently, the hyperplane $H(C,r)$ 
  separates $\cO(Y)$ from $K(\cR)$. 
\end{lemma}

\begin{proof}
  Since $Y\notin \mathcal{C}$, by~\Cref{rootedcircuits1} there
  exists a rooted circuit $(C,r)$ such that
  $C\cap Y=C\setminus \{ r\}$. Pick any point
  $y=(y_{1},\ldots,y_{n})$ from the $Y$-orthant $\cO(Y)$ of
  ${\mathbb R}^n$ (i.e., $y_e>0$ for $e\in Y$ and $y_e<0$ for
  $e\notin Y$). Since $C\cap Y=C\setminus \{ r\}$, we have $y_r<0$ and
  $y_e>0$ for any $e\in C\setminus \{ r\}$. Consequently, $y$ does not
  satisfy the inequality $n y_r>\sum_{e\in C\setminus \{ r\}} y_e$,
  and thus $y \notin K$.  As a result, each orthant $\cO(Y)$ with
  $Y\in \cC^*$ is separated from $K(\cR)$ by a hyperplane of the form
  $H(C,r)$ with $(C,r)\in \cR$. 
\end{proof}

Similarly, if $Y \in \cC \setminus \cC'$, then we can find a positive circuit 
$P \in \cP$ such that the hyperplane $H(P)$ separates the
$Y$-orthant $\cO(Y)$  from the polyhedron $K(\cP)$ (and thus from $K$ and $K_0$).

\begin{lemma}\label{lem-onlyCrealizedD}
  For any $Y \in \cC \setminus \cC'$, there exists a positive circuit 
  $P \in \cP$ such that $\cO(Y) \cap K(P) = \varnothing$. Consequently, the hyperplane $H(P)$ 
  separates $\cO(Y)$ from $K(\cP)$. 
\end{lemma}

\begin{proof}
  By~\Cref{lem-conflicts-def}, there exists a positive circuit 
  $P \in \cP$ such that $P \subseteq Y$. Consider a point
  $y=(y_{1},\ldots,y_{n})$ from the $Y$-orthant $\cO(Y)$ of
  ${\mathbb R}^n$. Since $P \subseteq Y$, we have $y_e>0$ for any
  $e\in P$. Consequently, $y$ does not satisfy the inequality
  $\sum_{e \in P} y_e < 0$, and thus $y \notin K(\cP)$.  As a result, each
  orthant $\cO(Y)$ with $Y\in \cC \setminus \cC'$ is separated from
  $K(\cP)$ by a hyperplane of the form $H(P)$ with $P\in \cP$.
\end{proof}

Finally, we show that if $Y \in \cC^*$, then the orthant $\cO(Y)$ does not 
intersect the polyhedron $K(\cR_0)$ defined by the critical rooted circuits 
of $\cA$. 
This follows from the following result. 

\begin{lemma}\label{lem-onlyCrealizedK0}
  For any $Y \in \cC^*$, there exists a rooted set $(A,r)$ and real
  numbers $(a_g)_{g \in A\setminus\{e\}}$ such that
  \begin{enumerate}[(1)]
  \item\label{cond-rooted-set}
    $Y\cap A = A \setminus \{r\}$,
  \item\label{cond-ag-petit} 
    $0 < a_g \leq \frac{1}{|Y|}$ for each $g \in A\setminus\{r\}$,
  \item\label{cond-hyperplan} for any $x \in K(\cR_0)$,
    $ x_r>\sum_{g\in A\setminus \{ r\}} a_{g}x_{g}$.
  \end{enumerate}
\end{lemma}

\begin{proof}
  In order to prove the lemma, we employ an \emph{elimination
    procedure} similar to Gaussian elimination. Note that the rooted
  sets $(A,r)$ we consider are not necessarily circuits. 
  We show the lemma by reverse
  induction on $|Y|$.

  The following claim ensures that the property holds for any
  $Y \in \cC^*$ such that $Y\cup \{ r\}\in \cC$ for some
  $r\in U\setminus Y$.
  
  \begin{claim}
    For any $Y \in \cC^*$, if there exists $r \in U\setminus Y$ such
    that $Y\cup\{r\} \in \cC$, then there exists a critical rooted circuit
    $(C,r)$ such that $C \subseteq Y$, and thus for each $x \in K(\cR_0)$,
    we have $x_r >\sum_{g\in C\setminus \{ r\}} \frac{1}{n}x_{g}$.
  \end{claim}

  \begin{claimproof}
      
    Consider a minimal subset $Z \subseteq Y$ such that $Z \in \cC^*$
    and $Z\cup\{r\} \in \cC$. Let $C = \ext(Z \cup \{r\})$ and note
    that $Z\cup\{r\} = \conv(C)$ by (C4).  By the definition of $C$, for any
    $z \in C$, we have $Z \cup \{r\}\setminus \{z\} \in \cC$, and by minimality of $Z$,
    we have $Z\setminus \{z\} \in \cC$. Consequently, $(C,r)$ is a critical rooted
    circuit of $\cC$ and thus by the definition of $K(\cR_0)$, for each
    $x \in K(\cR_0)$, we have
    $x_r >\sum_{g\in C\setminus \{ r\}} \frac{1}{n}x_{g}$.
  \end{claimproof}
  
  Suppose now that $Y \in \cC^*$ and that for any
  $e \in U\setminus Y$, we also have $Y\cup\{e\} \in \cC^*$.  By induction
  hypothesis, for any such set $Y\cup \{ e\}$, there exists a rooted set
  $(A_e,r_e)$ and coefficients $(a_g)_{g \in A_e\setminus\{r_e\}}$
  such that
  \begin{itemize}
  \item $(Y \cup \{e\}) \cap A_e = A_e\setminus \{r_e\}$,
  \item $0 < a_g \leq \frac{1}{|Y|+1}$ for each $g \in A_e\setminus\{r_e\}$,
  \item $x_{r_e}>\sum_{g\in A_{e}\setminus \{r_e\}} a_{g}x_{g}$ for
    any $x \in K(\cR_0)$.
  \end{itemize}

  If there exists $e \in U\setminus Y$ such that $e \notin A_e$, then
  $Y \cap A_e = A_e\setminus \{r_e\}$ and we are done. Consequently,
  we now assume that for any $e \in U\setminus Y$, we have
  $e \in A_e$, i.e., for any $e \in U\setminus Y$, we have
  $Y \cap A_e = Y\setminus\{e,r_e\}$.
  Consider the digraph $D_Y$ having $U\setminus Y$ as the vertex-set
  and where there is an arc $(e,f)$ precisely when $f = r_e$. Note
  that in $D_Y$ every vertex has an out-neighbor and consequently,
  $D_Y$ contains a cycle. Consider a shortest cycle
  $(e_0, e_1, \ldots, e_k,e_{k+1} = e_0)$ in $D_Y$. By induction
  hypothesis, for each $0\leq i \leq k$, there exists a rooted set
  $(A_{e_i},e_{i+1})$ and coefficients
  $(a_{i,g})_{g \in A_{e_i}\setminus\{e_{i+1}\}}$ such that
  \begin{itemize}
  \item $(Y \cup \{e_i\}) \cap A_{e_i} = A_{e_i} \setminus \{e_{i+1}\}$,
  \item $0 < a_{i,g} \leq \frac{1}{|Y|+1}$ for each
    $g \in A_e\setminus\{e_{i+1}\}$,
  \item for any $x \in K(\cR_0)$,
    $x_{e_{i+1}} > \sum_{g\in A_{e_i}\setminus \{ e_{i+1}\}}
    a_{i,g}x_{g} \quad = \quad a_{i,e_i}x_{e_i} + \sum_{g\in
      A_{e_i}\setminus \{ e_i,e_{i+1}\}} a_{i,g}x_{g}$.
  \end{itemize}


For each $0 \leq i \leq k$, let $b_i = a_{i,e_i}$ and observe that
$0< b_i \leq \frac{1}{|Y|+1} < 1$. Consequently,
$0 < \prod_{i=0}^{k} b_i < 1$.  For each $0 \leq i \leq k$ and each
$g \in Y\setminus A_{e_i}$, let $a_{i,g} = 0$. Observe that for any
$0 \leq i \leq k$, for every $x \in K(\cR_0)$, we thus obtain 
$x_{e_{i+1}} > b_ix_{e_i} + \sum_{g\in Y} a_{i,g}x_{g}$.
We have the following inequality:

\begin{claim}
  For every $0\leq i \leq k$ and  every point $x \in K(\cR_0)$, we have:
\begin{align*}
  x_{e_{i+1}} > \left(\prod_{j=0}^{i} b_j\right)x_{e_0} +\sum_{g\in Y} \left(
  \sum_{j=0}^{i} \left(\prod_{\ell=j+1}^{i} b_{\ell}\right)
  a_{j,g} \right) x_g.
\end{align*}
\end{claim}  
\begin{claimproof}
  We prove the claim by induction on $i$. The claim trivially holds
  for $i = 0$. Assume that the claim holds for $i<k$ and note that for
  any point $x \in K(\cR_0)$, we have:
\begin{align*}
  x_{e_{i+2}} &> b_{i+1}x_{e_{i+1}} + \sum_{g\in Y} a_{i+1,g}x_{g}\\
              &> b_{i+1}
                \left(\left(\prod_{j=0}^{i} b_j\right)x_{e_0}
                +\sum_{g\in Y} \left(
                \sum_{j=0}^{i} \left(\prod_{\ell=j+1}^{i} b_{\ell}\right)
                a_{j,g} \right) x_g\right) + \sum_{g\in Y} a_{i+1,g}x_{g}\\
              &= \left(\prod_{j=0}^{i+1} b_j\right)x_{e_0}
                +\sum_{g\in Y} \left(
                \sum_{j=0}^{i} \left(\prod_{\ell=j+1}^{i+1} b_{\ell}\right)
                a_{j,g} \right) x_g + \sum_{g\in Y} a_{i+1,g}x_{g}\\
              &= \left(\prod_{j=0}^{i+1} b_j\right)x_{e_0}
                +\sum_{g\in Y} \left(
                \sum_{j=0}^{i+1} \left(\prod_{\ell=j+1}^{i+1} b_{\ell}\right)
                a_{j,g} \right) x_g.
\end{align*}
\end{claimproof}

Consequently, for every point $x \in K(\cR_0)$, we have
\begin{align*}
  x_{e_0} = x_{e_{k+1}} & > \left(\prod_{j=0}^{k} b_j\right)x_{e_0} +\sum_{g\in Y} \left(
                          \sum_{j=0}^{k} \left(\prod_{\ell=j+1}^{k} b_{\ell}\right)
                          a_{j,g} \right) x_g.
\end{align*}
This implies the inequality
\begin{align*}
  \left(1-\prod_{j=0}^{k} b_j\right)x_{e_0}
                        &> \sum_{g\in Y} \left(
                          \sum_{j=0}^{k} \left(\prod_{\ell=j+1}^{k} b_{\ell}\right)
                          a_{j,g} \right) x_g,
\end{align*}
yielding
\begin{align*}
  x_{e_0} &> \frac{1}{1-\prod_{j=0}^{k} b_j}\sum_{g\in Y} \left(
            \sum_{j=0}^{k} \left(\prod_{\ell=j+1}^{k} b_{\ell}\right)
            a_{j,g} \right)x_g. 
\end{align*}

For any $g \in Y$, let
$a'_g = \frac{1}{1-\prod_{j=0}^{k} b_j} \sum_{j=0}^{k}
\left(\prod_{\ell=j+1}^{k} b_{\ell}\right)a_{j,g}$. Let
$C' \subseteq Y$ be the support of $(a'_g)_{g \in Y}$, i.e., the set
$\{g \in Y: a'_g \neq 0\}$.  Observe that the rooted set
$(C' \cup \{e_0\}, e_0)$ and the coefficients $(a'_g)_{g \in C'}$
satisfy Conditions~(\ref{cond-rooted-set}) and~(\ref{cond-hyperplan}) of the lemma.
Note that for each $0 \leq i \leq k$ and each $g \in Y$, we have 
$0< b_i \leq \frac{1}{|Y|+1}$ and $0 < a_{i,g} \leq
\frac{1}{|Y|+1}$. It is obvious that $a'_g > 0$ for any $g \in C'$.
In order to establish Condition~(\ref{cond-ag-petit}), it is then enough to show
that $a'_g \leq \frac{1}{|Y|}$ for any $g \in Y$. Let
$\Delta = \frac{1}{|Y|+1}$ and note that
$\frac{1}{1-\prod_{j=0}^{k} b_j} \leq
\frac{1}{1-\Delta^{k+1}}$. Consequently, for every $g \in Y$, we have

\begin{align*}
  a'_g &=   \frac{1}{1-\prod_{j=0}^{k} b_j}  \sum_{j=0}^{k} \left(\prod_{\ell=j+1}^{k}
         b_{\ell}\right)a_{j,g}
         \leq  \frac{1}{1-\Delta^{k+1}} \sum_{j=0}^{k} \left(\prod_{\ell=j+1}^{k}
         \Delta\right) \Delta = \frac{\Delta}{1-\Delta^{k+1}}\sum_{j=0}^{k} \Delta^{k-j} \\
       & = \frac{\Delta}{1-\Delta^{k+1}} \sum_{j=0}^{k} \Delta^{j}
         = \frac{\Delta}{1-\Delta^{k+1}}\cdot \frac{1-\Delta^{k+1}}{1-\Delta}
         = \frac{\Delta}{1-\Delta}
         = \frac{\frac{1}{|Y|+1}}{1-\frac{1}{|Y|+1}}
         = \frac{1}{|Y|+1} \cdot \frac{|Y|+1}{|Y|+1-1}\\
       & = \frac{1}{|Y|}, 
\end{align*}
which concludes the proof of the lemma. 
\end{proof}

\begin{proof}[Proof of~\Cref{realizations-convex-geometries}] 
  The proof that $(\cH_n,K)$ is a realization of  $\cI$  is obtained by combining Lemmas~\ref{lem-Crealized}, \ref{lem-onlyCrealizedK}, and~\ref{lem-onlyCrealizedD} and the equality $K=K(\cR)\cap K(\cP)$. 
  The proof that $(\cH_n,K_0)$ is a realization of $\cI$ is obtained by combining Lemmas~\ref{lem-Crealized}, \ref{lem-onlyCrealizedK0},
  and~\ref{lem-onlyCrealizedD} and the equality $K_0=K(\cR_0)\cap K(\cP)$. 
  Finally, the proof of the first assertion of the theorem that the pairs $(\cH_n,K(\cR))$ and
  $(\cH_n,K(\cR_0))$ are realizations of $\cA$ is a direct
  consequence of the second assertion because a convex geometry $\cA$
  is an ideal of itself not containing positive circuits (and thus $K=K(\cR)$ and
  $K_0=K(\cR_0)$). 
\end{proof}

\begin{example}
Consider the convex geometry $\cA = (U,\cC)$ on the left of~\Cref{fig:convsemigeo}. Here, $U = \{1,2,3,4\}$ and
$\cC = \{\varnothing, \{1\}, \{2\}, \{3\}, \{4\},$ $\{1,2\},
\{2,3\}, \{3,4\}$, $\{1,2,3\}, \{2, 3, 4\}, \{1,2,3,4\}\}$. The
rooted circuits of $\cA$ are
$(\{1,2,3\},2)$, $(\{2, 3, 4\},3)$, $(\{1,2,4\},$ $2)$, and $
(\{1,3,4\},3)$. Then in $\R^4$, $\cA$ is realized by the convex
set defined by the following inequalities:
\begin{align*}
  x_1+x_3 &< 4 x_2 & x_2+x_4 &<4 x_3 &  x_1+x_4 &<4 x_2 &  x_1+x_4 &<4 x_3 
\end{align*}

On the other hand, among the rooted circuits of $\cA$, only
$(\{1,2,3\},2)$ and $(\{2, 3, 4\},3)$ are critical. The two first
inequalities above correspond to these two critical rooted
circuits. These two inequalities are sufficient to have a realization
of $\cA$ in $\R^4$. Indeed, from the two first inequalities, we obtain
$4x_1 + x_4 < 15x_2$ and $x_1 + 4x_4 < 15x_3$. This implies that there
is no point in the realization corresponding to the sets
$\{1,4\}, \{1,2,4\}$, and $\{1,3,4\}$, that are precisely the subsets
of $U$ that are forbidden by the circuits $(\{1,2,4\},2)$, and
$ (\{1,3,4\},3)$.


\end{example}

From~\Cref{point-config} and~\Cref{pointshelling}, we obtain the following corollary. 
\begin{corollary}
    Any convex geometry $\cA = (U,\cC)$ admits a generalized convex shelling $\cC(U,Q_0)$ where $|Q_0|$ is the number of critical rooted circuits of $\cA$.
\end{corollary}

\subsection{Realization of convex geometries and convex dimension}
In the spirit of Kashiwabara et al.~\cite{KaNaOk}, Richter and
Rogers~\cite{RiRo} established that every convex geometry of convex
dimension $d:=\cdim(\cC)$ can be represented as a generalized convex shelling in
$\R^d$. The proof of the following result can be
seen as a dualization of the proof of~\cite{RiRo}.

\begin{theorem}\label{th-convgeom-cdim}
  A convex geometry $\cA = (U,\cC)$ of convex dimension $d:=\cdim(\cC)$ is
  realizable in $\R^d$.
\end{theorem}

\begin{proof}[Proof of \Cref{th-convgeom-cdim}]
Let $\cA = (U,\cC)$ be defined by a set of $d$ total
orders $(\leq_i)_{1\leq i \leq d}$. For each $e \in C$, let
$j_i(e)$ be the index of $e$ in $\leq_i$, (i.e., there are
precisely $j_i(e)$ elements $e' \in U$ such that $e' \leq_i e$).
Consider the sequence of integers $(a_j)_{0 \leq j \leq |U|}$ such
that $a_j=\frac{d+1}{d-1}(d^j -1)$ for each $0 \leq j \leq
|U|$. Observe that $a_{j+1} = da_j +d +1$ for any $0 \leq j < |U|$.
For each $e \in U$, let $b_i(e) = a_{j_i}(e)$ for each
$1 \leq i \leq d$ and consider the hyperplane
$H_e = \{x : \sum_{i=1}^d \frac{x_i}{b_i(e)} = 1\}$ and the halfspaces
$H_e^+ = \{x : \sum_{i=1}^d \frac{x_i}{b_i(e)} < 1\}$ and
$H_e^- = \{x : \sum_{i=1}^d \frac{x_i}{b_i(e)} > 1\}$.
We consider the hyperplane arrangement $\cH = \{H_e: e \in U\}$ and
its intersection with the positive orthant
$K =\{x : x_i>0, \forall 1 \leq i \leq d \}$. We assert that $(K,\cH)$
realizes $\cA = (U,\cC)$. The proof is a consequence of the two following
claims.

\begin{claim}\label{claimCy}
  For each $C \in \cC$, there exists $y \in K$ such that
  \begin{enumerate}[(i)]
  \item\label{claimCy1} for each $e \in C$, $y \in H_e^+$, i.e.,
    $\sum_{i=1}^d \frac{y_i}{b_i(e)} < 1$, and
  \item\label{claimCy2} for each $e \notin C$,  $y \in H_e^-$, i.e.,
    $\sum_{i=1}^d \frac{y_i}{b_i(e)} > 1$.
  \end{enumerate}
  
\end{claim}

\begin{claimproof}
  If $C = \varnothing$, let $y \in K$ such that $y_i = a_{|U|}+1$ for
  all $1 \leq i \leq d$. Observe that for each $e \in U$ and each
  $1 \leq i \leq d$, $b_i(e) \leq a_{|U|} < y_i$. Consequently, for
  each $e \in U$, $\sum_{i=1}^d \frac{y_i}{b_i(e)} > 1$, and
  $y \in H_e^-$. Thus the claim holds for $C = \varnothing$.
  
  Assume now that $C \in \cC \setminus \{\varnothing\}$. For each
  $1 \leq \ell \leq d$, let $e_{\ell} = \min_{\leq_{\ell}}
  C$. Observe that
  $U \setminus C = \bigcup_{\ell=1}^d \{e \in U: e \leq_{\ell}
  e_{\ell}\}$. Let $c_{\ell}= b_{\ell}(e_{\ell})$.  If
  $e_{\ell} = \min_{\leq_{\ell}} U$, let $c^-_{\ell}=0$ and if
  $e_{\ell} \neq \min_{\leq_{\ell}} U$, let $e^-_{\ell}$ be the
  predecessor of $e_{\ell}$ in $\leq_{\ell}$ and let
  $c^-_{\ell}= b_{\ell}(e^-_{\ell})$.
  Consider the point $y \in K$ such that $y_{\ell} = 1 + c_{\ell}^-$
  for each $1\leq \ell\leq d$.
  For each $e \in C$, note that $e_{\ell} \leq_{\ell} e$ for all
  $1 \leq \ell \leq d$. Consequently,
  $c_{\ell} =b_{\ell}(e_{\ell}) \leq b_{\ell}(e)$ for all
  $1 \leq \ell \leq d$.  Note that
  $c_{\ell} = d + 1 + d c_{\ell}^- = d (1+c_{\ell}^-) +1 > d
  y_{\ell}$. Consequently,
  $\sum_{\ell=1}^d \frac{y_{\ell}}{b_{\ell}(e)} \leq \sum_{\ell=1}^d
  \frac{y_{\ell}}{c_{\ell}} < \sum_{\ell=1}^d \frac{1}{d} = 1$,
  establishing~(\ref{claimCy1}).
  For each $e \notin C$, there exists $\ell$ such that
  $e \leq_{\ell} e^-_{\ell}$. Therefore
  $b_{\ell}(e) \leq b_{\ell}(e^-_{\ell}) = c_{\ell}^- <
  y_{\ell}$. Consequently, $\frac{y_{\ell}}{b_\ell(e)} > 1$ and since
  $\frac{y_i}{b_i(e)} > 0$ for all $1\leq i \leq d$ with
  $i \neq \ell$, we have $\sum_{i=1}^d \frac{y_i}{b_i(e)} > 1$,
  establishing~(\ref{claimCy2}).
\end{claimproof}

\begin{claim}\label{claimyC}
  Consider a point $x \in K$ such that
  $\sum_{i=1}^d \frac{x_i}{b_i(e)} \neq 1$ for all $e \in U$ and let
  $C = \{e \in U : \sum_{i=1}^d \frac{x_i}{b_i(e)} < 1\}$. Then
  $C \in \cC$.
\end{claim}

\begin{claimproof}
  If $C = \varnothing$, then $C \in \cC$ because $\varnothing$ is a convex
  of $\cC$. Assume now that $C \neq \varnothing$.  For each
  $1 \leq \ell \leq d$, let $e_\ell = \min_{\leq_{\ell}} C$.  To
  prove the claim, we establish that
  $C = \{e \in U: e_{\ell} \leq_{\ell} e , \forall 1 \leq \ell
  \leq d\}$.
  If $e \in C$, then for all $1 \leq \ell \leq d$,
  $e_\ell \leq_{\ell} e$ and thus
  $C \subseteq \{e \in U : e_{\ell} \leq_{\ell} e , \forall 1 \leq
  \ell \leq d\}$

  Conversely, suppose that $e_{\ell} \leq_{\ell} e$ for all
  $1 \leq \ell \leq d$.  If there exists $\ell$ such that
  $e = e_\ell$, then by definition of $e_\ell$, we have $e \in C$.
  Assume now that for every $1 \leq \ell \leq d$,
  $e_{\ell} \leq_{\ell} e$. For each $1 \leq \ell \leq d$, let
  $e_{\ell}^+$ be the successor of $e_{\ell}$ in $\leq_{\ell}$ and
  note that $e_{\ell}^+ \leq_{\ell} e$. Let
  $c_\ell = b_\ell(e_\ell)$ and let $c_\ell^+ = b_\ell(e_{\ell}^+)$.
  By the definition of $e_{\ell}$, $e_{\ell} \in C$ and thus
  $\sum_{i=1}^d \frac{y_i}{b_i(e_{\ell})} < 1$. Since $y \in K$, we
  have $\frac{y_i}{b_i(e_{\ell})} > 0$ for each $1 \leq i \leq d$ with
  $i \neq \ell$ and consequently,
  $\frac{y_{\ell}}{c_{\ell}} = \frac{y_{\ell}}{b_{\ell}(e_{\ell})} <
  1$.  Observe that $c_\ell < c_\ell^+ \leq b_{\ell}(e)$ for all
  $1 \leq \ell \leq d$.  Moreover, $c_\ell^+ = d c_\ell + d + 1$ and
  thus $b_{\ell}(e) \geq c_\ell^+ > d c_\ell$.  Consequently,
  $\sum_{\ell=1}^d \frac{y_{\ell}}{b_{\ell}(e)} < \sum_{\ell=1}^d
  \frac{y_{\ell}}{d c_{\ell}} =
  \frac{1}{d} \sum_{\ell=1}^d\frac{y_{\ell}}{c_{\ell}} < \frac{1}{d} d =
  1$. Therefore, $e \in C$ and thus we have
  $C = \{e \in U: e_{\ell} \leq_{\ell} e , \forall 1 \leq \ell
  \leq d\}$, establishing that $C \in \cC$.
\end{claimproof}

By~\Cref{claimCy}, for any convex set $C \in \cC$, there exists a
region corresponding to $C$ in $(K,\cH)$. By~\Cref{claimyC}, any
region of $(K,\cH)$ corresponds to a convex $C$ of $\cA = (U,\cC)$, concluding the proof of the theorem. 
\end{proof}

\subsection{Realization of ideals of convex geometries  and VC-dimension}

In this subsection, we consider the realizations of ideals $\cI=(U,\mathcal{C'})$ of convex geometries by minimizing or  bounding the dimension of the realizing space $\mathbb{R}^d$ by a dimension parameter of $\cI=(U,\mathcal{C'})$.  Since such $\cI=(U,\mathcal{C'})$ are ample \cite{BaChDrKo}, the VC-dimension 
$\VCdim(\cI)$ of $\cI$ coincides with the dimension of the largest cube of $\mathcal{C'}$. 
We are motivated by the following: 

\begin{corollary} \label{distributive}
For every distributive lattice $L$, we have $\VCdim(L)=\edim(L)$.
\end{corollary}
\begin{proof}
    The inequality $\VCdim(L)\leq \edim(L)$ follows from~\Cref{generalbounds}. For $\VCdim(L)\geq \edim(L)$ we use the fact the convex dimension and the VC-dimension of a distributive lattice coincide~\cite{EdJa} together with~\Cref{th-convgeom-cdim}.
\end{proof}

This leads us to believe:
\begin{conjecture} \label{ideal_dimension} The Euclidean dimension $\edim(\cI)$ of an ideal  $\cI=(U,\mathcal{C'})$ of a convex geometry 
is always upper bounded by a function of its VC-dimension $\VCdim(\cI)$. 
\end{conjecture} 

As a first approach we can get a bound for some situations:

\begin{corollary}\label{convgeom1}
  If $\cA=(U,\cC)$ is a bouquet of downset alignments with $\ell$ maxima, e.g., a median set systems with $\ell$ maxima, then $\cA$ is realizable and $\edim(\cA)\leq\min(|U|,\ell\VCdim(\cA))$.
\end{corollary}
\begin{proof}
  From~\Cref{ideal}, \Cref{co:mediansetsystem}, and
  \Cref{realizations-convex-geometries} we get realizability and the
  dimension bound of $|U|$. Now, using~\Cref{ideal} and
  \Cref{distributive} we obtain the bounds of the form
  $\ell\VCdim(\cA))$.
\end{proof}

In this subsection we confirm~\Cref{ideal_dimension} for trees and multisimplicial complexes (alias, ideals of the grid). Indeed, both our  bounds are best-possible.

A \emph{tree} is a connected acyclic graph. A rooted tree is a tree with a distinguished vertex $r$. A family of sets $\cA=(U,\cC)$ is a \emph{tree} if $\cA$ is a median system whose 1-inclusion graph is a tree.  

\begin{proposition}
    
 \label{trees}
Every tree $T$ can be realized in ${\mathbb{R}}^2$, i.e., $\edim(T)\leq 2$ holds.
\end{proposition}
\begin{proof}
Consider $T$ rooted at $r$. We will use the basic observation that $T$ can be constructed iteratively from the star of $r$ (i.e., $r$ with all its neighbors) by picking a leaf $\ell$ of the current tree and attaching to it all its children in $T$ at once. We describe a representation of $T$ following this construction sequence with the invariant  that every current leaf $\ell$ is represented by an isosceles  triangular cell whose base is in $\mathcal{H}$.

Let say $\deg(r)=k$.  Then we start with $K$ as a convex $k$-gon and we use $\cH$ to cut off, i.e., truncate, each vertex of $K$ such that an isosceles cell is created. Hence, we end up with a representation of the star of $r$ satisfying the invariant, i.e. there is a triangle with base in $\mathcal{H}$. If $\ell$ has $m$ children, then we add $m$ halfspaces to $K$ forming a convex curve between the intersection points of $\mathcal{H}$ and $K$, and staying inside the triangular region for $\ell$. Now, truncate each corner with a new element of $\mathcal{H}$ such that isosceles regions are created. We have represented the children of $\ell$ in the desired manner. See~\Cref{fig:trees} for an illustration.
\end{proof}

\begin{figure}[htp]
    \centering
    \includegraphics[width=0.8\textwidth]{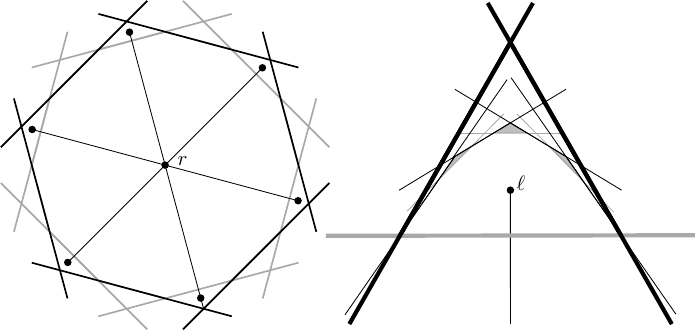}
    \caption{A two-dimensional Euclidean realization for trees.}\label{fig:trees}
\end{figure}

Median set systems are realizable by \Cref{convgeom1}. Since median systems are ample, their VC-dimension coincides with the dimension of a largest cube, and thus with their (topological) dimension. \Cref{distributive} and \Cref{trees} motivate the following strengthening of \Cref{ideal_dimension} in the case of median systems:

\begin{conjecture} Every median system of dimension $d$ admits a realization in ${\mathbb R}^{O(d)}$. 
\end{conjecture}


Recall that there exists a bijection between the ideals of the grid $(\mathbb{N}^n,\leq)$ and the multisimplicial complexes on a set $X$ of size $n$.

\begin{theorem}
Every ideal $L$ of $\mathbb{N}^n$ with $\VCdim(L)=d$ has a representation in ${\mathbb R}^{2d}$, i.e., $\edim(L)\leq 2\VCdim(L)$ holds.
\end{theorem}
\begin{proof}
Let $\VCdim(L)=r$. By Dilworth's Theorem, $L$ may be seen as a semilattice of ideals of a poset $X$ consisting of $n$ disjoint chains of some length $k$, such that each ideal intersects at most $r$ chains. Denote by $L_i$ the  sublattice of $L$ consisting of all the ideals of $L$ intersecting each chain in at most $i$ elements.

Let $P\subseteq \mathbb{R}^{2r}$ be the polar polytope of an $r$-neighborly polytope $P^*\subseteq \mathbb{R}^{2r}$ on $n$ vertices containing the origin. The latter exists~\cite{Gru03} and therefore $P$ has the property that for any set of $\ell\leq r$ facets of $P$, their intersection defines a face of dimension $r-\ell$ of $P$. Now consider the dilates $P_i:=iP$ for all $1\leq i\leq k$ and for each $i$, let $\mathcal{H}_i$ be the arrangement of facet-defining hyperplanes of $P_i$. Set $\mathcal{H}:=\bigcup_{i=1}^k\mathcal{H}_i$.

We will now proceed by induction on $i=1,\ldots,k$ by introducing halfspaces of the convex set $K$ and eventually changing the dilation for some $P_j$ with $j>i$ in such a way that for any $i$ we have $\mathcal{M}(\bigcup_{j=1}^i\mathcal{H}_j,K)=L_i$. We start with the basis case $i=1$.

\begin{case} $i=1$, i.e., $L_1$ is a simplicial complex.
\end{case}

\begin{claimproof}
First, note that every element of $L_1$ is represented by a cell of $\mathcal{H}_1$. Namely, if $x\in L_1$ corresponds to taking one element from each of the chains $C_{i_1}, \ldots C_{i_\ell}$ with $\ell\leq r$, then take the face $f$ of $P$ corresponding to the intersection of facets $f_{i_1}, \ldots f_{i_\ell}$. The corresponding facet defining hyperplanes $H_{i_1}, \ldots H_{i_\ell}$ define a cone over $f$ whose intersection with $P$ is $f$. This cell represents $x$. Since the polytope $P$ is simple (because it is the polar of the simplicial polytope $P^*$), in an $\epsilon$-neighborhood of $f$ we see the ideal of $x$ in $L_1$, where the minimum is represented by the interior of $P$.

Now, suppose that some  cell $C$ of $\mathcal{H}_1$ represents an ideal of $X$ that does not correspond to an element of $L_1$. Since $C$ and $P$ are open and convex we can choose a hyperplane $H$ that separates $C$ and $P$ and intersects the closure $\overline{P}$ of $P$ in $\overline{P}\cap \overline{C}$. Thus, if we add to $K$ the halfspace defined by $H$ that contains $P$, then we remove $C$ from the arrangement and all cells of $\mathcal{H}_1$ that correspond to faces not included in $\overline{P}\cap \overline{C}$ are still present. The only further cells of $\mathcal{H}_1$ that are removed are those that correspond to ideals of $X$ corresponding to supersets of the ideal corresponding do $C$. Since $L_1$ is an ideal, we did not want these regions. Consequently, taking as $K$ the intersection of all open halfspaces containing $P$ and defined by hyperplanes separating $P$ from cells $C$ representing ideals of $X$ that does not correspond to elements of $L_1$, we will obtain that $L_1=\mathcal{M}(\mathcal{H}_1,K)$.
\end{claimproof}

\begin{case} $i>1$.
\end{case}

\begin{claimproof}
By induction hypothesis, we can define the dilates $P_1,P_2,\ldots,P_{i-1}$ of $P$, their arrangements of supporting hyperplanes $\mathcal{H}_1,\mathcal{H}_2,\ldots,\mathcal{H}_{i-1}$, and the convex set $K$ in such a way  that  $\mathcal{M}(\bigcup_{j=1}^{i-1}\mathcal{H}_j,K)=L_{i-1}$. First, suppose that there exists an element $x\in L_{i}$ that is represented by a cell $C$ touching $P_i$ and that lies outside of $K$. Since $L$ is an ideal, the element $y\in L_{i-1}$ cutting the ideal of $x$ to height at most $i-1$ on each chain is in $L$. In particular, this point $p$ is represented by a cell touching $P_{i-1}$ within $K$. Hence, if $q$ is a point in the cell representing $x$, a line from $p$ to $q$ must cross the boundary of $K$. We can thus dilate $P_i$ to a smaller polytope such that the cell representing $x$ intersects the interior of $K$.
In order to avoid now some of the regions touching $P_i$ from the outside, we proceed as in the case of $P_1$ noting that all the hyperplanes that we add to $K$ are disjoint from the interior of $P_i$ hence they do not interfere with the so far obtained representation of $L_{i-1}$.
\end{claimproof}

This concludes the proof of the theorem. 
\end{proof}

The \emph{full simplicial complex} $U_{r,n}$ of dimension $r$ consists of all subsets of size at most $r$ of a set of size $n$. 

\begin{theorem}
For the full simplicial complex $U_{r,n}$ of dimension $r$ on a set $X=\{ f_1,\ldots,f_n\}$, we have  $\edim(U_{r,n})\geq \min(2r,n-1)$.
\end{theorem}
\begin{proof}
Let $\mathcal{M}(\mathcal{H},K)$ be a realization of $U_{r,n}$. We assume that $K$ is full-dimensional, otherwise we intersect with the affine hull of $K$. Consider the cell $P$ representing the empty set $\varnothing\in U_{r,n}$. Every 1-simplex of $U_{r,n}$ must be represented by a cell intersecting $P$ in a facet. More generally, for any $1\leq \ell\leq r$ every $\ell$-simplex  $\sigma=\{ f_{i_1}, \ldots, f_{i_{\ell}}\}$ must be represented by a cell that is separated  exactly by the facet-defining hyperplanes  $H_{i_1}, \ldots, H_{i_{\ell}}$ from $P$. Hence, the intersection of $H_{i_1}, \ldots, H_{i_{\ell}}$ must be a face of dimension  $r-\ell$ of $P$. Consider now the polyhedron $P'$ by removing all the facet-defining hyperplanes of $K$ from $P$. Hence, the polar polytope $P^*$ of $P'$ must be $r$-neighborly. It is known, that for all $n$, the smallest dimension in which an $r$-neighborly polytope $P^*$ on $n$ vertices exists is $2r$ or $P^*$ is a simplex and hence has dimension $n-1$, see e.g.~\cite{Gru03}. This lower bound carries over to $P'$, hence also to $P$. Hence this is a lower bound for the dimension of $K$.
\end{proof}

\section{Closing remarks}
\paragraph{Concepts of dimension.}
We have shown that the Euclidean dimension $\edim(\cA)$ of a convex geometry $\mathcal{A}=(U,\mathcal{C})$ lies between its VC-dimension $\VCdim(\cA)$ and the minimum of  $|U|$ and  $\cdim(\cA)$. It is known that convex geometries with bounded convex dimension and unbounded $|U|$ exist. On the other hand also there are convex geometries with $\cdim(\cA)$ exponential in $|U|$, see~\cite{knauer2023concepts}. Hence none of our bounds on $\edim(\cA)$ dominates the other. We believe that $\edim(\cA)$ can be bounded by a function of $\VCdim(\cA)$, see~\Cref{ideal_dimension}. Further, it would be interesting to compare $\edim(\cA)$ with other concepts of dimension for convex geometries, e.g.,  Dushnik-Miller dimension, Boolean
dimension, local dimension, and fractional dimension. Their behavior on convex geometries has been studied recently in~\cite{knauer2023concepts}. Another direction are representations by spheres and ellipsoids as studied in~\cite{AAN23}


\paragraph{From ideals to bouquets.}\label{ss-bouquets}

We believe that bouquets as a generalization of ideals deserve further investigation in their own right as a combinatorial structure. In the context of realizability we dare to state the following:

\begin{conjecture} Every bouquet of convex geometries is realizable. 
\end{conjecture}

\begin{example}
  The bouquet of convex geometries $(U,\cC')$ of the right of~\Cref{fig:convsemigeo} is realizable. The difference between
  $(U,\cC')$ and the convex geometry $(U,\cC)$ of the left of~\Cref{fig:convsemigeo} is that
  $\{1,2,3,4\} \in \cC \setminus \cC'$ and
  $\{1,4\} \in \cC' \setminus \cC$. Therefore, one cannot obtain a
  representation of $(U,\cC')$ by adding linear constraints to the
  representation of $(U,\cC)$.
  
  However, we claim that $(U,\cC')$ is realized in $\R^4$ by the
  convex set $K$ defined by the following two inequalities:
  \begin{align*}
    x_1+3x_3 &< x_2 & 3x_2+x_4 &<x_3 
  \end{align*}
  Observe that there is no point $x$ in the realization such that
  $x_1 > 0$, $x_3 > 0$, and $x_2 < 0$ (respectively, $x_2 > 0$,
  $x_4 > 0$, and $x_3 < 0$). Consequently, if $X$ is one of the sets
  $\{1,3\}$ or $\{1,3,4\}$ (respectively, $\{2,4\}$ or $\{1,2,4\}$),
  then $K$ does not contain any point of the $X$-orthant. Moreover, by
  adding the two inequalities, we get $x_1+2x_2+2x_3+x_4 <0$, and thus
  for $X = \{1,2,3,4\}$, there is no point $x \in K$ in the
  $X$-orthant. Therefore, the convex set $K$ does not intersect any
  $X$-orthant  such that $X \notin\cC'$.
  We now consider a set $X \in \cC'$ and we construct a point of the
  $X$-orthant that belongs to $K$. If $X$ contains at most two
  elements, let $x_i = 1$ for each $i \in X$. We can then find
  negative values for the remaining coordinates so that both equations
  are satisfied. Consider now the set $X = \{1,2,3\}$ (respectively,
  $\{2,3,4\}$) and the point $x = (1,5,1,-15) \in \R^4$ (respectively,
  $(-15,1,5,1) \in \R^4$). One can check that $x$ corresponds to $X$
  and satisfies both inequalities. Consequently, the bouquet of convex
  geometries $(U,\cC)$ is realizable.
\end{example}


\paragraph{Acknowledgment.} We are grateful to the referees for a careful reading and useful remarks that helped us to improve the presentation of the paper. 

\bibliographystyle{plainurl}
\bibliography{refs-semigeometries}

\begin{thebibliography}{10}

\bibitem{AAN23}
K.~Adaricheva, A.~Agarwal, and N.~Nevo.
\newblock Representation of convex geometries of convex dimension 3 by spheres.
\newblock Preprint, {arXiv}:2308.07384 [math.{CO}] (2023), 2023.
\newblock URL: \url{https://arxiv.org/abs/2308.07384}.

\bibitem{AW10}
Kira Adaricheva and Marcel Wild.
\newblock Realization of abstract convex geometries by point configurations.
\newblock {\em Eur. J. Comb.}, 31(1):379--400, 2010.
\newblock \href {https://doi.org/10.1016/j.ejc.2008.12.017}
  {\path{doi:10.1016/j.ejc.2008.12.017}}.

\bibitem{Av}
S.~P. Avann.
\newblock Metric ternary distributive semi-lattices.
\newblock {\em Proc. Amer. Math. Soc.}, 12(3):407 -- 414, 1961.
\newblock \href {https://doi.org/10.2307/2034206} {\path{doi:10.2307/2034206}}.

\bibitem{BaChDrKo}
H.-J. Bandelt, V.~Chepoi, A.~W.~M. Dress, and J.~H. Koolen.
\newblock Combinatorics of lopsided sets.
\newblock {\em European J. Combin.}, 27(5):669--689, 2006.
\newblock \href {https://doi.org/10.1016/j.ejc.2005.03.001}
  {\path{doi:10.1016/j.ejc.2005.03.001}}.

\bibitem{BaChKn}
H.-J. Bandelt, V.~Chepoi, and K.~Knauer.
\newblock {COM}s: Complexes of oriented matroids.
\newblock {\em J. Combin. Theory Ser. {A}}, 156:195--237, 2018.
\newblock \href {https://doi.org/10.1016/j.jcta.2018.01.002}
  {\path{doi:10.1016/j.jcta.2018.01.002}}.

\bibitem{BiKi}
G.~Birkhoff and S.~A. Kiss.
\newblock A ternary operation in distributive lattices.
\newblock {\em Bull. Am. Math. Soc.}, 53:749--752, 1947.
\newblock \href {https://doi.org/10.1090/S0002-9904-1947-08864-9}
  {\path{doi:10.1090/S0002-9904-1947-08864-9}}.

\bibitem{BjLVStWhZi}
A.~Bj{\"o}rner, M.~Las~Vergnas, B.~Sturmfels, N.~White, and G.~Ziegler.
\newblock {\em Oriented matroids.}, volume~46 of {\em Encycl. Math. Appl.}
\newblock Cambridge University Press, 2nd ed. edition, 1999.
\newblock \href {https://doi.org/10.1017/CBO9780511586507}
  {\path{doi:10.1017/CBO9780511586507}}.

\bibitem{BlLV}
R.~G. Bland and M.~Las~Vergnas.
\newblock Orientability of matroids.
\newblock {\em J. Comb. Theory, Ser. B}, 24(1):94--123, 1978.
\newblock \href {https://doi.org/10.1016/0095-8956(78)90080-1}
  {\path{doi:10.1016/0095-8956(78)90080-1}}.

\bibitem{BoRa}
B.~Bollob{\'{a}}s and A.~J. Radcliffe.
\newblock Defect {S}auer results.
\newblock {\em J. Combin. Theory Ser. {A}}, 72(2):189--208, 1995.
\newblock \href {https://doi.org/10.1016/0097-3165(95)90060-8}
  {\path{doi:10.1016/0097-3165(95)90060-8}}.

\bibitem{ChKnPh}
V.~Chepoi, K.~Knauer, and M.~Philibert.
\newblock Two-dimensional partial cubes.
\newblock {\em Electron. J. Comb.}, 27(3):research paper p3.29, 40, 2020.
\newblock \href {https://doi.org/10.37236/8934} {\path{doi:10.37236/8934}}.

\bibitem{Di}
B.~L. Dietrich.
\newblock A circuit set characterization of antimatroids.
\newblock {\em J. Comb. Theory, Ser. B}, 43(3):314--321, 1987.
\newblock \href {https://doi.org/10.1016/0095-8956(87)90007-4}
  {\path{doi:10.1016/0095-8956(87)90007-4}}.

\bibitem{Edel}
P.~H. Edelman.
\newblock Meet-distributive lattices and the anti-exchange closure.
\newblock {\em Algebra Univers.}, 10:290--299, 1980.
\newblock \href {https://doi.org/10.1007/BF02482912}
  {\path{doi:10.1007/BF02482912}}.

\bibitem{EdJa}
P.~H. Edelman and R.~E. Jamison.
\newblock The theory of convex geometries.
\newblock {\em Geom. Dedicata}, 19(3):247--270, 1985.
\newblock \href {https://doi.org/10.1007/BF00149365}
  {\path{doi:10.1007/BF00149365}}.

\bibitem{EdSa}
P.~H. Edelman and M.~E. Saks.
\newblock Combinatorial representation and convex dimension of convex
  geometries.
\newblock {\em Order}, 5(1):23--32, 1988.
\newblock \href {https://doi.org/10.1007/BF00143895}
  {\path{doi:10.1007/BF00143895}}.

\bibitem{ed-ma-82}
J.~Edmonds and A.~Mandel.
\newblock {\em Topology of Oriented Matroids}.
\newblock PhD thesis, University of Waterloo, 1982.
\newblock PhD thesis of A. Mandel, 333 pages.

\bibitem{FoLa}
J.~Folkman and J.~Lawrence.
\newblock Oriented matroids.
\newblock {\em J. Comb. Theory, Ser. B}, 25(2):199--236, 1978.
\newblock \href {https://doi.org/10.1016/0095-8956(78)90039-4}
  {\path{doi:10.1016/0095-8956(78)90039-4}}.

\bibitem{Gru03}
B.~Gr{\"u}nbaum.
\newblock {\em Convex polytopes. {Prepared} by {Volker} {Kaibel}, {Victor}
  {Klee}, and {G{\"u}nter} {M}. {Ziegler}}, volume 221 of {\em Grad. Texts
  Math.}
\newblock New York, NY: Springer, 2nd ed. edition, 2003.

\bibitem{HoMe}
U.~Hoffmann and K.~Merckx.
\newblock A universality theorem for allowable sequences with applications.
\newblock {\em CoRR}, abs/1801.05992, 2018.
\newblock URL: \url{http://arxiv.org/abs/1801.05992}, \href
  {https://arxiv.org/abs/1801.05992} {\path{arXiv:1801.05992}}.

\bibitem{KaNaOk}
K.~Kashiwabara, M.~Nakamura, and Y.~Okamoto.
\newblock The affine representation theorem for abstract convex geometries.
\newblock {\em Comput. Geom.}, 30(2):129--144, 2005.
\newblock \href {https://doi.org/10.1016/j.comgeo.2004.05.001}
  {\path{doi:10.1016/j.comgeo.2004.05.001}}.

\bibitem{KDiss}
K.~Knauer.
\newblock {\em Lattices and Polyhedra from Graphs}.
\newblock PhD thesis, Technische Universit\"at Berlin, 2010.
\newblock 131 pages.

\bibitem{KnMa}
K.~Knauer and T.~Marc.
\newblock On tope graphs of complexes of oriented matroids.
\newblock {\em Discrete Comput. Geom.}, 63(2):377--417, 2020.
\newblock \href {https://doi.org/10.1007/s00454-019-00111-z}
  {\path{doi:10.1007/s00454-019-00111-z}}.

\bibitem{knauer2023concepts}
K.~Knauer and W.~T. Trotter.
\newblock Concepts of dimension for convex geometries.
\newblock {\em SIAM J. Discrete Math.}, 38(2):1566--1585, 2024.
\newblock \href {https://doi.org/10.1137/23M1559853}
  {\path{doi:10.1137/23M1559853}}.

\bibitem{KoLo}
B.~Korte and L.~Lov{\'a}sz.
\newblock Shelling structures, convexity and a happy end.
\newblock Graph theory and combinatorics, {Proc}. {Conf}. {Hon}. {P}.
  {Erd{\H{o}}s}, {Cambridge} 1983, 219--232 (1984), 1984.

\bibitem{KoLoSch}
B.~Korte, L~Lov{\'a}sz, and R.~Schrader.
\newblock {\em Greedoids}, volume~4 of {\em Algorithms Comb.}
\newblock Springer, Berlin, 1991.
\newblock \href {https://doi.org/10.1007/978-3-642-58191-5}
  {\path{doi:10.1007/978-3-642-58191-5}}.

\bibitem{La}
J.~F. Lawrence.
\newblock Lopsided sets and orthant-intersection of convex sets.
\newblock {\em Pacific J. Math.}, 104(1):155--173, 1983.
\newblock \href {https://doi.org/10.2140/pjm.1983.104.155}
  {\path{doi:10.2140/pjm.1983.104.155}}.

\bibitem{Mne88}
N.~E. Mnev.
\newblock The universality theorems on the classification problem of
  configuration varieties and convex polytopes varieties.
\newblock Topology and geometry, {Rohlin} {Semin}. 1984-1986, {Lect}. {Notes}
  {Math}. 1346, 527-543 (1988)., 1988.

\bibitem{Mon}
B.~Monjardet.
\newblock The consequences of {D}ilworth's work on lattices with unique
  irreducible decompositions.
\newblock In K.~P. Bogart, R.~Freese, and J.~P.~S. Kung, editors, {\em The
  {D}ilworth Theorems: Selected Papers of Robert P. Dilworth}, pages 192--199.
  Birkh{\"a}user Boston, Boston, MA, 1990.
\newblock \href {https://doi.org/10.1007/978-1-4899-3558-8_17}
  {\path{doi:10.1007/978-1-4899-3558-8_17}}.

\bibitem{RiRo}
M.~Richter and L.~G. Rogers.
\newblock Embedding convex geometries and a bound on convex dimension.
\newblock {\em Discrete Math.}, 340(5):1059--1063, 2017.
\newblock \href {https://doi.org/10.1016/j.disc.2016.10.006}
  {\path{doi:10.1016/j.disc.2016.10.006}}.

\bibitem{Ric95}
J{\"u}rgen Richter-Gebert.
\newblock Mn{\"e}v's universality theorem revisited.
\newblock {\em S{\'e}min. Lothar. Comb.}, 34:15, 1995.
\newblock URL: \url{https://eudml.org/doc/119012}.

\bibitem{Shol}
M.~Sholander.
\newblock Medians, lattices, and trees.
\newblock {\em Proc. Am. Math. Soc.}, 5:808--812, 1954.
\newblock \href {https://doi.org/10.2307/2031872} {\path{doi:10.2307/2031872}}.

\bibitem{Sho91}
Peter~W. Shor.
\newblock Stretchability of pseudolines is {NP}-hard.
\newblock Applied geometry and discrete mathematics, {Festschr}. 65th
  {Birthday} {Victor} {Klee}, {DIMACS}, {Ser}. {Discret}. {Math}. {Theor}.
  {Comput}. {Sci}. 4, 531-554 (1991)., 1991.

\bibitem{Stern}
M.~Stern.
\newblock {\em Semimodular Lattices. Theory and Applications}, volume~73 of
  {\em Encycl. Math. Appl.}
\newblock Cambridge University Press, 1999.

\end{thebibliography}

\end{document}